\documentclass[12pt]{amsart}
\usepackage{graphicx}
\usepackage[centertags]{amsmath}
\usepackage{amsfonts}
\usepackage{amssymb}
\usepackage{amsthm}
\usepackage{comment}
\usepackage{fullpage}
\usepackage{url}

\theoremstyle{plain}
\newtheorem*{maintheorem*}{Main Theorem}
\newtheorem*{thm*}{Theorem}
\newtheorem*{thma*}{Theorem A}
\newtheorem*{thmaa*}{Theorem A'}
\newtheorem*{thmb*}{Theorem B}
\newtheorem*{thmo*}{Theorem 1.1}
\newtheorem*{thmc*}{Theorem C}
\newtheorem*{thmd*}{Theorem D}
\newtheorem*{thmf*}{Theorem 4.1}
\newtheorem*{remark*}{Remark}
\newtheorem{rem}{Remark}
\newtheorem*{conjecture*}{Conjecture}
\newtheorem*{prop*}{Proposition}
\newtheorem*{lem*}{Basic Lemma}
\newtheorem{thm}{Theorem}
\newtheorem{lem}{Lemma}[section]
\newtheorem{prop}[lem]{Proposition}
\newtheorem{cor}[lem]{Corollary}

\theoremstyle{definition}
\newtheorem{defn}{Definition}[section]

\numberwithin{equation}{section}

\def\bbq{\mathbb{Q}}
\def\bbr{\mathbb{R}}

\def\bbh{\mathbb{H}}

\def\tbf{\mathbf{t}}
\def\xbf{\mathbf{x}}
\def\ybf{\mathbf{y}}

\def\Dfrak{\mathfrak{D}}
\def\nxn{n_{\mathbf{x}}^-}

\def\vare{\varepsilon}
\def\g{\Gamma}
\def\h{\hspace{1mm}}
\def\hh{\hspace{.5mm}}

\def\d{{\rm{dist}}}
\def\homog{\g\backslash G}
\def\ba{\backslash}

\newcommand{\norm}[1]{\left\Vert#1\right\Vert}

\newcommand{\set}[1]{\left\{#1\right\}}

\newcommand{\pd}[2]{\frac{\partial #1}{\partial #2}}

\newcommand{\smatrix}[4]{\left(\begin{smallmatrix} #1 & #2\\ #3& #4 \end{smallmatrix}\right)}
\newcommand{\e}{\varepsilon}

\newcommand{\vphi}{{\varphi}}

\newcommand{\calC}{\mathcal{C}}

\newcommand{\calF}{\mathcal{F}}

\newcommand{\calO}{\mathcal{O}}

\newcommand{\bbZ}{\mathbb{Z}}

\newcommand{\bbQ}{\mathbb{Q}}
\newcommand{\bbR}{\mathbb R}
\newcommand{\bbC}{\mathbb C}

\newcommand{\bbN}{\mathbb N}

\newcommand{\SL}{ \mathrm{SL}}
\newcommand{\SO}{ \mathrm{SO}}
\newcommand{\Sl}{ \mathrm{sl}}
\newcommand{\su}{ \mathrm{su}}
\newcommand{\PSL}{ \mathrm{PSL}}

\newcommand{\SU}{ \mathrm{SU}}

\newcommand{\bs}{\backslash}
\newcommand{\id}{1\!\!1}

\newcommand{\frakD}{\mathfrak{D}}

\newcommand{\frakg}{\mathfrak{g}}

\DeclareMathOperator{\Res}{Res}

\newcommand{\lims}{\mathop{\overline{\lim}}}

\title{Logarithm laws for one parameter unipotent flows}
\author{Dubi Kelmer and Amir Mohammadi}%
\address{Department of Mathematics, 301 Carney Hall
Boston College, Chestnut Hill, MA 02467-3806.}
\email{dubi.kelmer@bc.edu}
\address{Department of Mathematics
The University of Texas at Austin
1 University Station, C1200
Austin, TX 78750}
\email{amir@math.utexas.edu}

\begin{document}

\maketitle

\begin{abstract}
We prove logarithm laws and shrinking target properties for unipotent flows on the homogenous space $\Gamma\bs G$ with $G=\SL_2(\bbR)^{r_1}\times\SL_2(\bbC)^{r_2}$ and $\Gamma\subseteq G$ an irreducible non-uniform lattice.
Our method relies on certain estimates for the norms of (incomplete) theta series in this setting.
\end{abstract}

\section*{Introduction}
Let $G$ denote a semisimple Lie group and $\Gamma\subseteq G$ a non-uniform irreducible lattice. Consider the action of an unbounded one parameter subgroup $\{g_t|t\in \bbR\}\subseteq G$ on the space $X=\Gamma\bs G$ endowed with the probability $G$-invariant Haar measure $\sigma$. By Moore's Ergodicity Theorem this action is ergodic on $\homog$ and hence the orbit of $\sigma$-a.e. $x\in X$ becomes equidistributed. In particular, these orbits make excursions far out into the cusps. A natural way of measuring the rate of these excursions is considering the distances $\d(o,xu_t)$ from a fixed point $o\in X$ and asking what is the fastest rate at which they grow for a typical point $x\in X$; note that the ergodicity of the action implies $\lims_{t\to\infty}\d(o,xg_t)=\infty$ for $\sigma$-a.e. $x\in X$.

This problem can be treated as an instance of a shrinking target problem and, as such, an upper bound for the rate of excursions follows from the Borel-Cantelli Lemma. We recall that the Borel-Cantelli Lemma implies that for any sequence $\{B_\ell\}_{\ell\in \bbN}$ of shrinking targets in $X$, if $\sum_{\ell}\sigma(B_\ell)<\infty$, then for $\sigma$-a.e. $x\in X$ the set $\{\ell|xg_\ell\in B_\ell\}$ is finite. If we also assume that the events $xg_\ell\in B_\ell$ and $x g_k\in B_k$ are independent (i.e,  $\sigma(B_\ell g_{-\ell}\cap B_k g_{-k})=\sigma(B_\ell)\sigma(B_k)$), then the converse is also true. In general, the converse is not true without the independence assumption. However, it may hold under additional assumptions on the shrinking sets. A family $\mathcal{B}$ of subsets of $X$ is called a Borel-Cantelli family for $\{g_\ell\}$ if the converse holds for all sets from $\mathcal{B}$. That is, $\mathcal{B}$ is Borel-Cantelli for $\{g_\ell\}$ if for any countable collection of shrinking targets $\{B_\ell\}\subset \mathcal{B}$ with $\sum_{\ell} \sigma(B_\ell)=\infty$, for $\sigma$-a.e. $x\in X$ the set $\{\ell|xg_\ell\in B_\ell\}$ is infinite.

In our setting, we consider targets shrinking to infinity given by
$$B_{r}=\{x\in X|\d(o,x)>r\}.$$
Under an appropriate normalization of the distance function we have that $\sigma(B_r)\asymp e^{-r}$ and hence, the first half of the Borel-Cantelli Lemma together with a standard continuity argument imply that
$\lims_{t\to \infty} \frac{\d(o,xg_t)}{\log(t)}\leq 1$ for $\sigma$-a.e. $x\in X$.
If the family of neighborhoods of infinity $\mathcal{B}=\{B_r|r>0\}$ is Borel-Cantelli for $\{g_\ell\}$ then this bound is sharp
and the flow  $\{g_t\}_{t\in\bbR}$ satisfies the logarithm law, $\lims_{t\to \infty} \frac{\d(o,xg_t)}{\log(t)}=1$ for $\sigma$-a.e. $x\in X$.
We note that if this holds for a lattice $\Gamma$, then it also holds for any commensurable lattice $\Gamma'$ (as the difference between the corresponding distance functions is uniformly bounded).

The case of the geodesic flow on finite volume non-compact hyperbolic manifolds (that is, $\Gamma\backslash \bbh^{m+1}$ with $\Gamma\subseteq \SO(m+1,1)$ a non-uniform lattice) was studied by Sullivan~\cite{Su}. Sullivan utilized a geometric proof Khinchin's theorem on approximation of reals by rationals to prove that the family $\mathcal{B}=\{B_r|r>0\}$ is Borel-Cantelli for the geodesic flow.

The general case of locally symmetric spaces of noncompact type and $\{g_t|t\in \bbR\}$ a one parameter diagonalizable subgroup was proved by Kleinbock and Margulis~\cite{KM2}. Using the exponential rate of mixing of such flows they show that the events $xg_\ell\in B_{r_\ell}$ and $xg_{k}\in B_{r_{k}}$ are (exponentially close to being) independent. Then, using an effective version of the second half of the Borel-Cantelli Lemma they proved that the family $\mathcal{B}$ is Borel-Cantelli for these flows, and hence, diagonalizable flows satisfy the logarithm law. In fact, for this result one can use other distance like function $\Delta:X\to [0,\infty)$, instead of the function $\Delta(x)=\d(o,x)$, where we call a function distance like if (after appropriate normalization) it satisfies $\sigma\{x|\Delta(x)\geq r\}\asymp e^{-r}$.

More recently the case of one parameter unipotent groups has attracted some attention. For unipotent flows, the (polynomial) rate of mixing is not fast enough to obtain the desired effective independence used in the case of diagonalizable flows. Nevertheless, in ~\cite{AM} Athreya and Margulis proved logarithm laws for one parameter unipotent groups acting on $X_n=\SL_n(\bbZ)\bs \SL_n(\bbR)$, with respect to a distance like function $\alpha_1:X_n\to[0,\infty)$, given in terms of the length of the shortest vector in $\Lambda=\bbZ^ng$. To obtain this result they use the interpretation of this space as the space of lattices in $\bbR^n$ and prove a random version of Minkowski's theorem, showing that a large set in $\bbR^n$ intersects most lattices (with respect to normalized Haar measure on $X_n$).

If one considers the action of the full horospherical group (rather than a one parameter unipotent group) it is possible to get much sharper results. Indeed, in \cite{Athreya11} Athreya studied the action of full horospherical group on $X_n$. In this setting, he was able to give a precise result for the rate of excursions for every $x\in X_n$, in terms of certain Diophantine properties of $x$.
In the special case where $G=\SL_2(\bbR)$ the horospherical group is a one parameter group, and hence in these cases the logarithm laws follow from this analysis (for any lattice $\Gamma\subseteq \SL_2(2,\bbR)$).

In this paper, we generalize the approach in ~\cite{AM} to prove logarithm laws for one parameter unipotent flows on more general homogenous spaces $\Gamma\bs G$.
Though this approach should work in general (at least for lattices of $\bbQ$-rank one) it relies on estimates of certain theta functions that we were able to establish so far only for the case where $\Gamma$ is an irreducible lattice in $G=\SL_2(\bbR)^{r_1}\times \SL_2(\bbC)^{r_2}$.

Before we state our results we introduce the following notation. We write $A\lesssim B$ or $A=O(B)$ to indicate that $A\leq cB$
for some constant $c$.  If we wish to emphasize that constant depends on
some parameters we use subscripts, for example
$A\lesssim_\epsilon B$. We also write $A\asymp B$ to indicate that
$A\lesssim B\lesssim A$.

We can now state our main result.
\begin{thm}\label{t:logarithmlaw}
Let $G=\SL_2(\bbR)^{r_1}\times \SL_2(\bbC)^{r_2}$, $\Gamma\subseteq G$ an irreducible lattice, and $K\subset G$ a maximal compact. Let $\d$ denote a distance function on $X=\Gamma\bs G$ obtained from a left $G$-invariant, bi $K$-invariant Reimannian metric on $G$, normalized so that $\sigma(B_r)\asymp e^{-r}$.
Then, for any one-parameter unipotent group $\{u_s\}_{s\in\bbr}\subseteq G$
\begin{equation}\label{e:logarithmlaw} \forall o\in X,\; \mbox{ for $\sigma$-a.e. } x\in X,\quad  \lims_{s\rightarrow\infty}\frac{\d(o,xu_s)}{\log s}=1.\end{equation}
\end{thm}

\begin{rem}
If we assume that $\Gamma$ is arithmetic, then our proof implies the stronger statement that the family $\mathcal{B}$ of neighborhoods of infinity is Borel-Cantelli for $\{u_\ell\}$.
\end{rem}

For the reader's convenience we give a brief outline of the proof.
As a first step we note that it is enough to find for every $\epsilon>0$ a set $Y=Y_\epsilon$ of positive measure such that
 \begin{equation}\label{e:Yeps}
\forall x\in Y, \quad \lims_{t\rightarrow\infty}\frac{\d(o,xu_t)}{\log t}\geq 1-\epsilon.\end{equation}
Indeed, the set
$\{x|\lims_{t\rightarrow\infty}\frac{\d(o,xu_t)}{\log t}\geq 1-\epsilon\}$
is invariant under the flow and hence, from ergodicity, if it has positive measure it must have full measure.
Next, in order to construct such a set, it is enough to find a sequence of sets $Y_k\subseteq X$ satisfying that their measures are uniformly bounded from below, and that
\begin{equation}\label{e:Yk}
\forall \;x\in Y_k \; \exists\;  \ell\geq k\mbox{ such that } \frac{\d(o,xu_\ell)}{\log \ell}\geq 1-\epsilon.
\end{equation}
Indeed, in that case the set $Y=\cap_{\ell=1}^\infty \cup_{k=\ell}^\infty Y_k$ will have positive measure and satisfy \eqref{e:Yeps}.

Finally, we construct the sets $Y_k$ explicitly by taking appropriate unions of translations of neighborhoods of the cusp at infinity.
To describe this construction we need some additional notation. Let $P\subseteq G$ denote the maximal parabolic subgroup of upper triangular matrices and let $\Gamma_\infty=\Gamma\cap P$ denote the stabilizer of the cusp at infinity. Let $Q\subseteq P$ be the maximal subgroup containing $\Gamma_\infty$ such that $\Gamma_\infty\bs Q$ is relatively compact; see \eqref{e:Q} below. In section \ref{s:construct} we construct an explicit sequence of sets  $\Dfrak_k\subseteq Q\bs G$ such that the sets $Y_{\Dfrak_k}$ satisfy \eqref{e:Yk} where for any $\frakD\subseteq Q\bs G$ we let
$$Y_{\Dfrak}=\{\Gamma g\in \Gamma\bs G|\; Q\gamma g\in \Dfrak \mbox{ for some } \gamma\in \Gamma\}.$$
Moreover, it follows from our construction that $|\frakD_k|\to \infty$ where $|\frakD_k|$ denotes the measure of $\Dfrak_k$ with respect to the Haar measure on $Q\bs G$. It is worth mentioning that this is where unipotent flow is used most crucially. Indeed if one carries the same construction for diagonalizable flows then $|\frakD_k|$ will remain ``small". The reason we get large measure sets for unipotent flow is that the $K$-parts of the unipotent group in the Iwasawa decomposition, $G=KP,$ are ``far" from each other. This prevents large overlaps of the translate of a small region and results in large measure sets, see Section~\ref{sec;vol} for more details.

In order to complete the proof, all we need to show that the sets $Y_{\Dfrak_k}$ constructed above do not have measures shrinking to zero.
A standard way to control the measure of sets $Y_\frakD\subseteq \Gamma\bs G$ obtained from sets $\Dfrak\subseteq Q\bs G$ by lifting to $G$ and then folding back into $\Gamma\bs G$, comes from analysis of a corresponding incomplete theta series. Specifically,
to any compactly supported function $f\in C_c(Q\bs G)$ the corresponding theta function\footnote{These functions have different names in the literature. They are sometimes referred to as incomplete theta series as well as incomplete Eisenstein series or pseudo Eisenstein series. In this work we will use the (not so standard but shorter term) theta function.}, $\Theta_f\in L^2(\Gamma\bs G)$, is defined by
\begin{equation*}
\Theta_f(g)=\sum_{\gamma\in \Gamma_\infty\bs \Gamma}f(\gamma g).
\end{equation*}
Note that if $f$ is supported on $\Dfrak\subseteq Q\bs G$, then $\Theta_f$ is supported on $Y_{\Dfrak}$. In order to exclude the possibility that $\sigma(Y_{\Dfrak})\to 0$ while $|\Dfrak_k|\to\infty$, it is enough to show that the theta functions corresponding to the indicator functions of $\Dfrak_k$ do not become to large. To do this we bound the growth of $\norm{\Theta_{\id_{\Dfrak}}}$ in terms of $|\frakD|$, where $\norm{\Theta_f}^2=\int_{\Gamma\bs G}|\Theta_f(g)|^2d\sigma(g)$. In fact, we prove the following general bound for the $L^2$ norm of these theta functions that is of independent interest.
\begin{thm}\label{t:thetabound}
Let $G=\SL_2(\bbR)^{r_1}\times \SL_2(\bbC)^{r_2}$ and $\Gamma\subseteq G$ an arithmetic irreducible lattice. Then 
for all positive $f\in C^\infty_c(Q\bs G)$
\begin{equation}\label{e:thetabound}\norm{\Theta_f}^2\lesssim_\Gamma\norm{f}_2^2+\norm{f}_1^2,\end{equation}
where the norms on the right are with respect to Haar measure on $Q\bs G$.
\end{thm}

It is worth mentioning that this is the same strategy used in~\cite{AM}. There, they get an estimate, similar to the one in Theorem~\ref{t:thetabound}, using a result of C.~A.~Rogers~\cite{Rogers56}, based on so called spherical symmetrization of the theta function above. However, the result of Rogers is quite specific to the case that is handled  in~\cite{AM} and does not seem to generalize to other cases; in particular it does not generalize to the case in hand.

Our proof of Theorem \ref{t:thetabound} relies on a formula for $\norm{\Theta_f}^2$ in terms of the poles of corresponding Eisenstein series (see Proposition  \ref{p:ThetaMain}), together with a comparison of the norms of theta functions constructed with respect to different lattices. In particular, for $\Gamma=\SL_2(\calO_K)$ the Eisenstein series has no exceptional poles and the bound \eqref{e:thetabound} easily follows from this formula. Next,
we show that if the bound \eqref{e:thetabound} holds for $\Gamma$, then it holds for any finite index subgroup (see Lemma \ref{l:SG}). Since any arithmetic lattice is commensurable to $\SL_2(\calO_K)$ we can prove Theorem \ref{t:thetabound} for all arithmetic lattices.

We note that the assumption that $\Gamma$ is arithmetic is probably not needed. When $r_1+r_2\geq 2$ any irreducible lattice is arithmetic so this is only an issue when $G=\SL_2(\bbR)$ or $G=\SL_2(\bbC)$. In those cases, even without the arithmeticity assumption, we prove \eqref{e:thetabound} for a specific family of positive functions, approximating the indicator functions of $\frakD_k$.
Specifically, for $G=\SL_2(\bbR)$ or $\SL_2(\bbC)$ we consider the family of functions $f^{(\lambda)}\in C^\infty_c(Q\bs G)$ for $\lambda\in [1,\infty)$ as follows (see section \ref{s:coordinates} for the coordinate used for this definition).
\begin{itemize}
\item When $G=\SL_2(\bbR)$ we let
$$f^{(\lambda)}(n_xa_{t}k_{\theta})=v_\lambda(t)\psi(\lambda \theta),$$
where $v_\lambda(t)$ approximates the indicator function of $[-(1+\epsilon)\log(\lambda),0]$ and $\psi\in C^\infty_c(\bbR)$ is a fixed smooth compactly supported function.
\item When $G=\SL_2(\bbC)$ we let
$$f^{(\lambda)}(n_xa_{t}k_{\theta,\alpha,\beta})=v_\lambda(2t)\psi(\lambda \sin(\theta),\lambda(\alpha-\beta)),$$
where $v_\lambda(t)$ approximates the indicator function of $[-(3+\epsilon)\log(\lambda),0]$ and $\psi\in C^\infty_c(\bbR^2)$ is a fixed smooth compactly supported function. 
\end{itemize}

\begin{thm}\label{t:ThetaR}
For any lattice $\Gamma\subseteq G$ with $G=\SL_2(\bbR)$ or $\SL_2(\bbC)$ and $f^{(\lambda)}$ as above,
\[\norm{\Theta_{f^{(\lambda)}}}^2\lesssim_{\Gamma,\psi,\epsilon} \norm{f^{(\lambda)}}_2^2+\norm{f^{(\lambda)}}_1^2,\]
uniformly for $\lambda\in [1,\infty)$.
\end{thm}

\begin{rem}
It is an interesting question for what groups $G$ and lattices $\Gamma$ can one show the bound \eqref{e:thetabound} for all theta functions.
We note that a formula for the norms of theta functions in terms of poles of Eisenstein series can be obtained in general for arithmetic lattices in semi-simple algebraic groups (see \cite{Harishchandra68}). Moreover, the fact that the bound \eqref{e:thetabound} for a lattice $\Gamma$ implies the same bound (with perhaps a different constant) for any finite index subgroup, can also be proved in this generality. The main ingredient that is missing in order to prove Theorem \ref{t:thetabound}, and hence also Theorem \ref{t:logarithmlaw}, for (arithmetic) hyperbolic manifolds in higher dimensions, is the existence of a nice family of lattices for which the Eisenstein series is known not to have exceptional poles.
\end{rem}

%
%

\subsection*{Acknowledgments}
The second named author would like to thank G. Margulis and A. Salehi Golsefidy for many
useful and illuminating conversations. The first named author would like to thank S. Yu for pointing out a mistake in a previous version of this paper. The first named author was partially supported by NSF grant DMS-1001640.

\section{Notation and preliminaries}\label{sec;notation}

\subsection{Coordinates}\label{s:coordinates}
Let $G=\prod_{j=1}^n G_j$ with $G_j=\SL_2(\bbR)$ for $j\leq r_1$ and $G_j=\SL_2(\bbC)$ for $j>r_1$.
We fix coordinates on $G$ and use them to fix a normalization of the Haar measure.

On each $G_j$ we have a decomposition $G_j=N_jA_jK_j$ with
$N_j$ upper triangular, $A_j$ diagonal with real coefficients and $K_j$ maximal compact (so $K_j=\SO(2)$ for $j\leq r_1$ and $K_j=\SU(2)$ for $j>r_1$).
We denote by $N=\prod_j N_j,\;A=\prod_j A_j$, and $K=\prod_j K_j$. Define the group $M\subseteq G$ to be the centralizer of $A$ in $K$, so that $M=\prod_j M_j$ with $M_j=\{\pm 1\}$ for $j\leq r_1$ and $M_j$ is the group of diagonal unitary matrices for $j>r_1$.

For ${\bf{t}}\in \bbR^n$ and $\mathbf{x}\in \bbR^{r_1}\times\bbC^{r_2}$ we denote by $a_{\bf{t}}=(a_{t_1},\ldots,a_{t_n})\in A$ and $n_{\bf{x}}=(n_{x_1},\ldots,n_{x_n})\in N$, where $a_{t}=\begin{pmatrix}e^{t/2} &0\\  0 &e^{-t/2}\end{pmatrix}\in A_j$ and $n_{x}=\begin{pmatrix}1 & x \\ 0 &1 \end{pmatrix}\in N_j$.
Let $\mu=(\mu_1,\ldots,\mu_n)$ with $\mu_{j}=1$ for $j\leq r_1$ and $\mu_j=2$ for $j>r_1$.
We fix once and for all a normalization for the Haar measure of $G$ such that in the coordinates $g=n_{\bf{x}}a_{\bf{t}}k$ we have
\[dg=\exp(-\sum_j\mu_j t_j)d{\bf{t}}d{\bf{x}}dk.\]
where $dk$ is normalized to be a probability measure on $K$.
Namely, for $j\leq r_1$ and $k_\theta=\smatrix{\cos(\theta)}{\sin(\theta)}{-\sin(\theta)}{\cos(\theta)}\in \SO(2)$ we have $dk_\theta=\frac{d\theta}{2\pi}$ and for $j>r_1$ and $k_{\theta,\alpha,\beta}=\smatrix{\cos(\theta)e^{i\alpha}}{\sin(\theta)e^{i\beta}}{-\sin(\theta)e^{-i\beta}}{\cos(\theta)e^{-i\alpha}}$,
we have $dk_{\theta,\alpha,\beta}=\tfrac{1}{16\pi^2}|\sin(2\theta)|d\theta d\alpha d\beta$.

\subsection{Coordinates at the cusp}
Let $G$ be as above and let $\Gamma\subseteq G$ denote an irreducible lattice.
We say that $\Gamma$ has a cusp at infinity if $\Gamma_\infty=\Gamma\cap P$ is non trivial where $P\subseteq G$ denotes the group of upper triangular matrices.

From the work of Shimizu \cite{Shimizu63} we have that a typical element of $\Gamma_\infty$ is of the form
\[\begin{pmatrix} u & \alpha\\0 & u\end{pmatrix}=\left(\begin{pmatrix} u_1 & \alpha_1\\0 & u_1^{-1}\end{pmatrix},\ldots,\begin{pmatrix} u_n & \alpha_n\\0 & u_n^{-1}\end{pmatrix}\right),\]
with $\prod_j|u_j|^{\mu_j}=1$. In fact, if $n\geq 2$ then $\Gamma$ is arithmetic and hence commensurable to $\SL_2(\calO_K)$ with $K$ a number field with $r_1$ real places and $r_2=n-r_1$ (pairs of) complex places and $\calO_K$ is the ring of algebraic integers. In particular, in that case (perhaps after conjugating $\Gamma$) we may assume that $\alpha\in \calO_K$ is an algebraic integer and $u\in \calO_K^*$ is in the group of units.

We introduce new coordinates that are more suitable for working with the cusp at infinity.
Consider the lattice $\calO_\Gamma\subseteq \bbR^{r_1}\times\bbC^{r_2}$ defined by
$$\calO_\Gamma=\{\xbf\in \bbR^{r_1}\times\bbC^{r_2}|\; n_\xbf\in \Gamma_\infty\}.$$
and let $U_{\Gamma}\subseteq \bbR^n$ denote the image of the homomorphism $\iota:\Gamma_\infty\to \bbR^n$ given by
\[\iota(\begin{pmatrix} u & \alpha\\ 0 & u^{-1}\end{pmatrix})=(\log(|u_1|),\ldots,\log(|u_n|)).\]
Then $\calO_\Gamma\subseteq \bbR^{r_1}\times\bbC^{r_2}$ is a lattice (of real rank $r_1+2r_2$) and
$U_{\Gamma}\subseteq \{x\in \bbR^n| \sum_j\mu_j x_j=0\}$ is a lattice of rank $r_1+r_2-1=n-1$.

Fix an integral basis $v_1,\ldots,v_{n-1}$ of $U_\Gamma$ and complete it to a basis of $\bbR^n$ by adding the vector $\eta=\frac{1}{n}(\mu_1^{-1},\ldots,\mu_n^{-1})$.
\begin{defn}\label{d:regulator}
We define the regulator $R_\Gamma$ of $\Gamma$ as the determinant of the matrix
\begin{equation}\label{e:Matrix}
D=\begin{pmatrix}
v_{1,1} & \cdots & v_{n-1,1} & \eta_1\\
\vdots & \ddots & \vdots & \vdots\\
v_{1,n} & \cdots & v_{n-1,n} & \eta_n
\end{pmatrix}.
\end{equation}
This does not depend on our choice of basis. When $\Gamma=\PSL_2(\calO_K)$ then $R_{\Gamma}=2^{-r_2}R_{\calO_K}$ with $R_{\calO_K}$ the regulator of $\calO_K$.
\end{defn}
We define our coordinates at the cusp as $g({\bf{x}},{\bf{t}},k)=n_{\bf{x}}a_{D\bf{t}}k$.
Note that if $\tilde{\bf{t}}=D\bf{t}$ then $\sum_j\mu_j \tilde{t}_j=t_n$, consequently, in these coordinates the Haar measure is
$dg=R_\Gamma e^{-t_n}d\textbf{t}d\textbf{x}dk.$

\subsection{Cusp neighborhood}\label{s:cuspneighborhood}
Using the coordinates at the cusp is easy to see that the set
\begin{equation}\label{e:Cusp}
\calF_{\infty}=\{n_{\bf{x}}a_{D\bf{t}}k|{\bf{x}}\in \calF_{\calO_{\Gamma}},\; {\bf{t}}\in [-1,1]^{n-1}\times\bbR,\;k\in K\}\end{equation}
with $\calF_{\calO_{\Gamma}}\subseteq \bbR^{r_1}\times\bbC^{r_2}$ denoting a fundamental domain for $\calO_{\Gamma}$, is a fundamental domain for $\Gamma_\infty\bs G$ (to be precise, when $r_1=0$ it is possible that $\Gamma_\infty\cap K\neq \{1\}$ in which case we need to replace $K$ by a fundamental domain for $(\Gamma_\infty\cap K)\bs K$).

For $\tau\in \bbR$ we define the cusp neighborhoods
\begin{equation}\label{e:CuspT}\calF_{\infty}(\tau)=\{n_{\bf{x}}a_{D\bf{t}}k\in\calF_\infty|t_n\geq \tau\}.\end{equation}
We now give a coordinate free description of these sets. For this, let $A^1\subseteq A$ denote the one parameter group generated by $a_\eta$ with $\eta=\tfrac{1}{n}(\mu_1^{-1},\ldots,\mu_n^{-1})$ as above,
and let $Q\subseteq P$ denote the subgroup given by
\begin{equation}\label{e:Q}
Q=\left\{\begin{pmatrix} a & x \\ 0 & a^{-1}\end{pmatrix}\in G| a,x\in \bbR^{r_1}\times \bbC^{r_2},\; \prod_{j=1}^n |a_j|^{\mu_j}=1\right\}.
\end{equation}
Note that $\Gamma_\infty\subseteq Q$ and the quotient $\Gamma_\infty\bs Q$ is relatively compact. In fact, using the coordinates at the cusp we see that the (relatively compact) set
\begin{equation}\label{e:FQ}\omega_{\Gamma}=\{n_{\bf{x}}a_{D\bf{t}}m\in Q| {\bf{x}}\in \calF_{\calO_{\Gamma}},\;m\in M,\; {\bf{t}}\in [-1,1]^{n-1}\times\{0\}\}\end{equation}
is a fundamental domain for $\Gamma_\infty\bs Q$.
Using this notation we can write the cusp neighborhood as the Siegel set $\calF_\infty(\tau)=\omega_{\Gamma} A^1(\tau) K$ where
$A^1(\tau)=\{a_{\eta t}\in A^1|t\geq \tau\}$.

\subsection{Cusp decomposition}
Let $\Xi=\{\xi_1,\ldots,\xi_h\}$ denote a complete set of inequivalent cusps of $\Gamma$ (that is, $\xi_j^{-1}P\xi_j\cap\Gamma$ is not trivial and $\xi_i\xi_j^{-1}\not\in \Gamma$) and let $\calF_j(\tau)=\xi_j\calF_\infty(\tau)$ denote the corresponding cusp neighborhoods. From Shimizu's Lemma \cite[Lemma 5]{Shimizu63} we have that for $\tau$ sufficiently large $\gamma \calF_i\cap\calF_j=\emptyset$ for any $\gamma\in \Gamma$ and any $1\leq i,j\leq h$.
Moreover, for such $\tau$ we have a Siegel fundamental domain of the form
\[\calF_{\Gamma}=\mathfrak{C}\cup \calF_1(\tau)\cup\cdots\cup \calF_h(\tau),\]
with $\mathfrak{C}$ relatively compact, satisfying that $\Gamma \calF_{\Gamma}=G$, and that the set $\{\gamma\in \Gamma| \gamma \calF_{\Gamma}\cap \calF_{\Gamma}\neq \emptyset\}$ is finite. Indeed, for $n\geq 1$ we have that $\Gamma$ is arithmetic and this follows from the reduction theory of Borel and Harish-Chandra~\cite{BHC}. When $n=1$, $\Gamma$ might not be arithmetic, however, in this case $G=\SL_2(\bbR)$ or $\SL_2(\bbC)$ and this is well known (see e.g., \cite[Proposition 3.9]{ElstrodtGrunewaldMennicke98}).

Let $\d$ denote a distance function on $X$ coming from a left $G$-invariant bi $K$-invariant Riemannian metric on $G$. Specifically, denote by
$\d_G$ the distance function on $G$ induced from the Riemannian metric and let $\d=\d_{X}$ denote the distance function on $X=\homog$ given by
$$\d(\Gamma g,\Gamma h)=\inf_{\gamma\in\Gamma}\d_G(g,\gamma h).$$
Clearly, $\d(x,y)\leq\d_G(g,h)$ for any choice of representatives $x=\Gamma g, y=\Gamma h$; the converse, where the representatives $g,h$ are taken from a Siegel set is known as Siegel's conjecture. It's proof is due to Ding \cite{Ding94} for the case of $G=\SL_n(\bbr)$ and to Leuzinger \cite{Leuzinger04} and Ji \cite{Ji98}, independently, for a general group and an arithmetic lattice. The case at hand however is simpler, as $\Gamma$ is either arithmetic of $\bbq$-rank one, or that $G$ is of real rank one; in these cases there is no gap in the original proof of Borel~\cite[Theorem C]{B}.
Applying these results to our setting we get the following
\begin{lem}\label{l:borel-2}
For $o\in \calF_\Gamma$ and $\tau_0\in \bbR$ fixed, there exists a constant $C$, such that for any $g\in \calF_j(\tau),\; j=1,\ldots, h$ with $\tau>\tau_0$ and any $\gamma\in\Gamma$,
\begin{equation}\label{e;metric-fund} \d_G(o,\gamma  g)\geq \d_G(o,g)-C\end{equation}
\end{lem}
In particular, this implies that for any $g\in \calF_j(\tau)$ with $\tau>\tau_0$ we have
$$\d_G(o,g)-C\leq \d(o, \Gamma g)\leq \d_G(o,g).$$
Moreover, any $g\in \calF_j(\tau)$ can be written as $g=\xi_j q a_{\eta t} k$ with $q\in \omega, k\in K$ and $t\geq \tau$. Since $\omega$ is relatively compact we have that
$d_G(o,g)=d_G(o,a_{\eta t})+O(1)$. Consequently, we have that any $x=\Gamma g\in \Gamma\bs G$ with $g=\xi_j q a_{\eta t} k\in \calF_j(\tau)$ satisfies
\begin{equation}\label{e:cuspdist}\d(o,x)=\d_G(o,a_{\eta t})+O(1).\end{equation}
\begin{rem}
We note that Theorem \ref{t:logarithmlaw} (with the same proof) holds for a more general distance like function $\Delta$ (instead of the standard $\Delta(x)=\d(o,x)$), as long as it can be evaluated on Siegel sets in the sense of \eqref{e:cuspdist} and behave nicely under the right action of $K$ in the sense that $\Delta(xk)=\Delta(x)+O(1)$ uniformly for $k\in K$.
\end{rem}

%
%

\subsection{Normalization}
We normalize the Haar measure $\sigma=\sigma_\Gamma$ to be a probability measure on $\Gamma\bs G$.
That is, we set $d\sigma(g)=\frac{dg}{v_\Gamma}$ with $v_\Gamma=\int_{\Gamma\bs G}dg$.

We also fix compatible normalization of the Haar measures on $Q$ and on $Q\bs G$. First, we identify $Q\bs G= M\bs A^1K$ and we normalize the Haar measure on $Q\bs G$ so that for any $f\in C_c(Q\bs G)$ lifted to a $Q$-invariant function on $G$ we have
\begin{equation}\label{e:QHaar}\int_{Q\bs G}f(g)dg=\int_{\bbR}\int_{M\bs K}f(a_{\eta t}k)e^{-t}dtdk.\end{equation}
We normalize the Haar measure on $Q$ so that for any compactly supported function $f$ on $G$ we have
\[\int_G f(g)dg=\int_{Q\bs G}\int_Q f(qg)dqdg.\]
Specifically, in the coordinates, $q=n_{\xbf}a_{D\tbf}m$ with $\tbf\in \bbR^{n-1}\times \{0\}$ we have $dq=R_{\Gamma}d\tbf d\xbf dm$.
For future reference we note that, with this normalization, the fundamental domain \eqref{e:FQ} has measure
\begin{equation}\label{e:volQ}
|\omega_{\Gamma}|=\int_{\omega_{\Gamma}}dq=2^{n-1}R_{\Gamma}|\calF_{\calO_{\Gamma}}|.
\end{equation}

Finally, we fix a normalization for the metric $\d_G$.
For any $r>0$ let
\begin{equation}\label{e:cuspball}
B_r=\{x\in \Gamma\bs G|\d(o,x)\geq r\}.
\end{equation}
We recall that by \cite[Section 5]{KM2} for any metric, $\d$, on $\Gamma\bs G$ arising as above from a metric, $\d_G$, there is a constant $k>0$ such that $\sigma(B_{r})\asymp e^{-kr}$.
We can always re-normalize the metric $\d_G$ to make $k=1$, and this is precisely the normalization we fix. That is, with this normalization we have that
\begin{equation}\label{e:cuspvol}
\sigma(B_{r})\asymp e^{-r}.
\end{equation}
Note that if $\d(o,x)$ is sufficiently large, then $x$ has a unique representative $x=\Gamma g$ with $g$ in one of the cusp neighborhoods $\calF_j(\tau)$. Consequently, from \eqref{e:cuspdist} and \eqref{e:cuspvol} one sees that this normalization of $\d_G$ imply that
\begin{equation}\label{e:distcusp}
\d_G(o,a_{\eta t})=|t|+O(1).
\end{equation}

\section{Theta functions}
Let $G$ be as above and $\Gamma\subseteq G$ an irreducible lattice with a cusp at infinity.
To any smooth compactly supported function $f\in C_c(\Gamma\bs G)$ we attach the theta function
\begin{equation*}\label{e:thetadefn}
\Theta_f(g)=\sum_{\gamma\in \Gamma_\infty\bs \Gamma}f(\gamma g),
\end{equation*}
where the sum is over a full set of representatives for $\Gamma_\infty\bs\Gamma$.
We will sometimes consider theta functions of the same $f\in C^\infty_c(Q\bs G)$ constructed with respect to different lattices. We write $\Theta_f^{\Gamma}(g)$ when we want to emphasize the dependence on the lattice.

Note that (since $f$ is compactly supported) the infinite sum over $\Gamma_\infty\bs \Gamma$ is actually a finite sum. In particular it converges pointwise to a continuous function. Also, since $f$ is already invariant under $\Gamma_\infty$, the resulting function is invariant under $\Gamma$. We can thus think of $\Theta_f$ as a function on $\Gamma\bs G$ and we define its norm by $\norm{\Theta_f}^2=\int_{\Gamma\bs G}|\Theta_f(g)|^2d\sigma_\Gamma(g)$.

This section is devoted to the proof of Theorem \ref{t:thetabound} and Theorem \ref{t:ThetaR},
that is, to prove the bound $\norm{\Theta_f}^2\lesssim_{\Gamma} \norm{f}_1^2+\norm{f}_2^2$,
where the norms $\norm{f}_1^2$ and $\norm{f}_2^2$ are taken with respect to Haar measure on $Q\bs G$. As a first step, we use the (standard) unfolding trick to get the following useful identity
\begin{lem}\label{l:unfolding}
For $\Theta_f$ as above and any $F\in L^2(\Gamma\bs G)$
\begin{equation*}
\int_{\Gamma\bs G}\Theta_f(g)F(g)d\sigma(g)=\int_{\Gamma_\infty\bs G}f(g)F(g)d\sigma(g).
\end{equation*}
\end{lem}
\begin{proof}
Let $\calF_\Gamma\subseteq G$ denote a fundamental domain for $\Gamma\bs G$. Note that $\calF_\infty=\bigcup_{\gamma\in \Gamma_\infty\bs\Gamma} \gamma^{-1}\calF_\Gamma$ is a fundamental domain for $\Gamma_\infty\bs G$, hence
\begin{eqnarray*}\lefteqn{
\int_{\calF_{\Gamma}}\Theta_f(g) F(g)d\sigma(g)= \sum_{\gamma\in \Gamma_\infty\bs \Gamma}\int_{\calF_{\Gamma}}f(\gamma g)F(g)d\sigma_(g)}\\
&&= \sum_{\gamma\in \Gamma_\infty\bs \Gamma}\int_{\gamma^{-1}\calF_{\Gamma}}f(g)F(g)d\sigma(g)
= \int_{\calF_\infty}f(g)F(g)d\sigma(g).
\end{eqnarray*}
\end{proof}
In particular, taking $F=\overline{\Theta}_f=\Theta_{\bar{f}}$ we get
\begin{equation}\label{e:unfold2}
\norm{\Theta_f}^2=\int_{\Gamma_\infty\bs G}\overline{f(g)} \Theta_f(g)d\sigma(g).
\end{equation}
This identity (together with the fact that $f$ is compactly supported and $\Theta_f$ is continuous) shows that $\norm{\Theta_f}$ is finite so that indeed $\Theta_f\in L^2(\Gamma\bs G)$.

Next, in order to get an estimate on this norm, we express $\Theta_f$ itself as an integral over Eisenstein series;
but before we can do that, we need to recall some of the theory of Eisenstein series.
 For details we refer to \cite{Efrat87,Hejhal83,Sarnak83} in the spherical case, to \cite{Warner79} in the non-spherical real rank one case, and to \cite{Harishchandra68} in general.

\subsection{Spherical Eisenstein series}
For each of the factors $G_j=\SL_2(\bbR)$ or $\SL_2(\bbC)$ let $\Omega_j$ denote the Casimir operator, this is a second order differential operator acting on $C^\infty(G_j)$ commuting with the left action of $G_j$ (see e.g., \cite[Page 198]{Lang75} and \cite[Page 62]{JorgensonLang08} for explicit formulas for these operators on $SL_2(\bbR)$ and $\SL_2(\bbC)$ respectively).

For any $s\in \bbC$, define the function $\vphi_{s}\in C^\infty(Q\bs G/K)$ by
\begin{equation}\label{e:vphi}
\vphi_{s}(na_{\bf{t}}k)=\exp(s\sum_j\mu_jt_j).\end{equation}
This is a joint eigenfunction of the Casimir operators $\Omega_{G_j}$ with eigenvalues $\mu_j^2s(1-s)$ respectively;
note that $\vphi_{1-s}$ is also an eigenfunction with the same eigenvalues.
In the cusp coordinates this function looks like
\begin{equation}\label{e:vphi2}
\vphi_s(na_{D\bf{t}}k)=\vphi_s(a_{\eta t_n})=e^{st_n}.
\end{equation}

Given a lattice $\Gamma$ (with a cusp at infinity), the spherical Eisenstein series (at infinity) is defined by
\[E(s,g)=E_\Gamma(s,g)=\sum_{\gamma\in \Gamma_\infty\bs \Gamma} \vphi_{s}(\gamma g).\]
This series absolutely converges for $\Re(s)>1$, it is right $K$-invariant and (since the operators $\Omega_j$ commute with the left action of $G$) it is also a joint eigenfunction of the Casimir operators with the same eigenvalues. The constant term of the Eisenstein series is defined as
\[E^o(s,g)=\frac{1}{|\calF_{\calO_\Gamma}|}\int_{\xbf\in \calF_{\calO_\Gamma}} E(s,n_\xbf g)d\xbf,\]
and satisfies
\begin{equation}\label{e:constant}
E^o(s,g)=\vphi_{s}(g)+\calC_\Gamma(s)\vphi_{1-s}(g).
\end{equation}
The function $\calC(s)=\calC_\Gamma(s)$ can be continued to a meromorphic function having no poles on the half plane $\Re(s)\geq 1/2$, except for possibly finitely many simple poles in the interval $(\tfrac{1}{2},1]$ (see \cite[Proposition 6.1]{Efrat87}). The residues at the exceptional poles are all positive, moreover, the residue at $s=1$ is related to the volumes of the fundamental domains as follows (see, e.g. \cite[Lemma 2.15]{Sarnak83})
\begin{equation}\label{e:residue}
\Res_{s=1}(\calC(s))=\frac{|\omega_{\Gamma}|}{v_{\Gamma}}=\frac{R_{\Gamma}2^{n-1}|\calF_{\calO_\Gamma}|}{v_{\Gamma}}.
\end{equation}
In the case where there are $h>1$ cusps, the function $\calC(s)=\calC_{1,1}(s)$ is one of the diagonal coefficients of the scattering matrix $\Phi(s)$. The scattering matrix itself has a meromorphic continuation and satisfies the functional equation $\Phi(s)\Phi(1-s)=I$ and $\Phi(s)^*=\Phi(\bar{s})$.
In particular, on the critical strip $\Re(s)=\tfrac12$ the scattering matrix  is unitary and hence $|\calC(s)|\leq 1$ for $\Re(s)=\tfrac{1}{2}$.
\begin{rem}\label{r:normalizatoin}
We note that our normalization of the Eisenstein series is not the standard one for $\SL_2(\bbC)$. In particular, in our normalization the critical strip is $0\leq \Re(s)\leq 1$ rather than the standard $0\leq \Re(s)\leq 2$. However, we find that this choice is more suitable for working with products of $\SL_2(\bbR)$ and $\SL_2(\bbC)$ simultaneously. The reader should be cautioned that whenever comparing to the literature dealing with $\SL_2(\bbC)$ one should replace $s$ by $2s$.
\end{rem}
\begin{rem}\label{r:CongruenceConstant}
When $\Gamma=\SL_2(\calO_K)$ the constant term (respectively the determinant of the scattering matrix) can be expressed explicitly in terms of the Dedekind Zeta function of $K$ (respectively, the Class field of $K$); see \cite{Efrat87, EfratSarnak85}. In particular, in these cases there are no poles in the half plane $\Re(s)\geq \tfrac{1}{2}$ except for the pole at $s=1$.
\end{rem}

\subsection{Raising and lowering operators}
We will need to consider also Eisenstein series that are not right $K$-invariant. Using suitable raising and lowering operators we can obtain all the information we need from the spherical case. We briefly recall how to construct these operators for $G=\SL_2(\bbR)$ and $G=\SL_2(\bbC)$ separately.

We start with the (simpler) case where $G=\SL_2(\bbR)$ and $K=SO(2)$.
For any $m\in \bbZ$ let $\phi_m\in L^2(M\bs K)$ denote the function $\phi_m(k_\theta)=e^{2im\theta}$ lifted to a function on $G$ by setting $\phi_m(nak)=\phi_m(k)$. We say that a function $\vphi$ on $\SL_2(\bbR)$ is of $K$-weight $m$ if it satisfies $\vphi(gk)=\vphi(g)\phi_m(k)$.

For $s\in \bbC$ we let $\vphi_s$ denote the function on $\SL_2(\bbR)$ given by
$\vphi_s(u_xa_t k)=e^{st}$ and let $\vphi_{s,m}(g)=\vphi_s(g)\phi_m(g)$. Then $\vphi_{s,m}$ and $\vphi_{1-s,m}$ are the unique functions on $N\bs G$ of $K$-weight $m$ that are eigenfunctions of the Casimir operator of $\SL_2(\bbR)$ with eigenvalue $s(1-s)$.

Let $\pi$ denote the right regular representation of $G$. We also denote by $\pi$ the corresponding representation of the Lie algebra $\frakg=\Sl_2(\bbR)$ on $C^\infty(G)$. Fix a basis $h=\begin{smatrix} 1 0 0 {-1}\end{smatrix},\;e=\begin{smatrix} 0 1 0 0\end{smatrix},\; f=\begin{smatrix}0 0 1 0\end{smatrix}$ for $\frakg$ and define raising and lowering operators by
\[a^{\pm}=\pi(h)\pm i(\pi(e)+\pi(f)).\]
These operators send a vector of $K$-weight $m$ to a vector of $K$-weight $m\pm 1$. In particular, $a^{\pm}\vphi_{s,m}$ is of weight $m\pm 1$ and is also an eigenfunction with the same eigenvalue. In particular, we get that $a^{\pm}\vphi_{s,m}$ is a scalar multiple of $\vphi_{s,m\pm1}$.

To compute this scalar (and its dependence on $s$ and $m$) we identify $C^\infty(N\bs G)\cong C^\infty(\bbR^2)$ (with $G$ acting on the right).
In this realization, with the coordinates
\[a_t k_\theta \mapsto (-\sin(\theta)e^{-t/2},\cos(\theta)e^{-t/2})=(x_1,x_2),\]
we have that
$\vphi_{s,m}=\frac{(x_1+ix_2)^{2m}}{(x_1^2+x_2^2)^{s+m}}$ and the raising and lowering operators are given by
\begin{equation}\label{e:ladderR}
a^\pm=x_1\pd{}{x_1}-x_2\pd{}{x_2}\pm i(x_1\pd{}{x_2}+x_2\pd{}{x_1}).
\end{equation}
With these formulas it is not hard to check that
\begin{equation}\label{e:raising1}
a^{\pm}\vphi_{s,m}=-2(s \pm m)\vphi_{s,m\pm 1}.
\end{equation}

We now treat the case $G=\SL_2(\bbC)$; here, $K=\SU(2)$ and $M$ is the group of diagonal unitary matrices.
Consider the representation of $K$ on $L^2(M\bs K)$ acting on the right.
We say that a vector $\vphi\in L^2(M\bs K)$ is of $M$-weight $\ell$ if it satisfies
\[\vphi(k \smatrix{e^{i\theta}}{0}{0}{e^{-i\theta}})=e^{2i\ell\theta}\vphi(k).\]
We can decompose $L^2(M\bs K)$ into irreducible invariant subspaces
\[L^2(M\bs K)=\bigoplus_{m=0}^\infty L^2(M\bs K,m),\]
where $L^2(M\bs K,m)$ is isomorphic to the irreducible representation of $\SU(2)$ of dimension $2m+1$ and it contains a unique (up to scalar multiplication) vector of $M$-weight $l$ for any $|l|\leq m$. We further note that every irreducible representation of $\SU(2)/\{\pm I\}$ occurs in this decomposition exactly once.

Let $\pi$ denote the right representation of $G$ on $L^2(M\bs G)$. We say that a vector $\vphi\in L^2(M\bs G)$ is of $M$-weight $\ell$ if it is of $M$-weight $\ell$ for the restriction of the representation to $K$ (i.e, if $\vphi(g \smatrix{e^{i\theta}}{0}{0}{e^{-i\theta}})=e^{2i\ell\theta}\vphi(g)$).

We use the basis $h,e,f,ig,ie,if$ for $\mathfrak{g}=\Sl_2(\bbC)$ (where $h,e,f$ are as above)
and note that  $ih,w=e-f,v=i(e+f)$ is a basis for the subspace $\su(2)\subseteq \Sl_2(\bbC)$.
We then have that a vector $\vphi$ is of $M$-wight $\ell$ if $\pi(ih)\vphi=2l\vphi$ (where $\pi$ denotes the representation of $\Sl_2(\bbC)$ on $C^\infty(M\bs G)$ corresponding to the right regular representation).

We define the following raising and lowering operators.
\[a^{\pm}=\tfrac{1}{2}(\pi(e)\mp i\pi(ie)),\; b^{\pm}=\tfrac{1}{2}(\pi(f)\pm i\pi(if)),\; a_K^{\pm}=a^{\pm}-b^{\pm}.\]
By looking at the commutation relation with $\pi(ih)$ we see that each of these operators sends a vector of $M$-weight $\ell$ to a vector of $M$-weight $\ell\pm 1$ (or to zero). Moreover, the operator $a_K^\pm=\tfrac{1}{2}(\pi(w)\mp i\pi(v))$ preserves $K$-invariant spaces.

We say that a vector $\vphi\in C^\infty(M\bs G)$ is a vector of highest weight $m$ if it is of $M$-weight $m$ and $a_K^+v=0$ (such a vector is contained in an irreducible $K$-invariant subspace isomorphic to $L^2(M\bs K,m)$). We also note that if $\vphi$ is of highest weight $m$, then $a^+\vphi=b^+\vphi$ is either zero or a vector of highest weight $m+1$ (this is also a direct consequence of the commutation relation).

For any $m\in\bbN$ let $\phi_m\in L^2(M\bs K,m)$ denote a vector of highest weight $m$ and extend it to a function on $G$ by $\phi(nak)=\phi(k)$. Let $\vphi_{s}\in C^\infty(N\bs G/K)$ (extended to a function on $G$) be defined by
$\vphi_{s}(na_tk)=e^{2st}$, and let $\vphi_{s,m}(g)=\vphi_s(g)\phi_m(g)$.
Then $\vphi_{s,m},\vphi_{1-s,m}\in C^\infty(N\bs G)$ are both eigenfunction of the Casimir operator with eigenvalue $4s(1-s)$ and are vectors of highest weight $m$. Consequently $a^+\vphi_{s,m}=b^+\vphi_{s,m}$ is a vector of highest weight $m+1$ and is also an eigenfunction with the same eigenvalue, and hence, a scalar multiple of $\vphi_{s,m\pm 1}$.

In order to compute this scalar we identify $N\bs G\cong \bbC^2$ and use the coordinates $z_1,\bar{z}_1,z_2,\bar{z}_2$ on $\bbC^2\cong \bbR^4$.
In these coordinates, with $G$ acting on the right on $C^\infty(N\bs G)\cong C^\infty(\bbC^2)$ the rasing and lowering operators are given by
\begin{equation}\label{e:ladderC}
a^+=z_1\pd{}{z_2},\quad a^-=\bar{z}_1\pd{}{\bar{z}_2},\quad b^+=\bar{z}_2\pd{}{\bar{z}_1},\quad b^-=z_2\pd{}{z_1},
\end{equation}
the Casimir operator is
\begin{equation}\label{e:CasimirC}
\Omega=(\bar{z}_1\pd{}{\bar{z}_1}+\bar{z}_2\pd{}{\bar{z}_2})^2+2(\bar{z}_1\pd{}{\bar{z}_1}+\bar{z}_2\pd{}{\bar{z}_2})
\end{equation}
and $\vphi_{s,m}=c_m\frac{(z_1\bar{z}_2)^m}{(|z_1|^2+|z_2|^2)^{2s+m}}$. Using \eqref{e:ladderC} we get that
\begin{equation}\label{e:raising2}
a^+\vphi_{s,m}=\kappa_m(2s+m)\vphi_{s,m+1},
\end{equation}
where $\kappa_m\neq 0$ is a constant depending only on $m$ (but not on $s$).

\subsection{Non-spherical Eisenstein series}
We now go back to the general setting where $G=\prod_jG_j$ with $G_j=\SL_2(\bbR)$ for $j\leq r_1$ and $G_j=\SL_2(\bbC)$ for $j>r_1$ and define the non-spherical Eisenstein series (cf. \cite[Chapter II section 2]{Harishchandra68}).

We decompose the representation of $K$ given by the right action on $L^2(M\bs K)$ into irreducible components
\[L^2(M\bs K)=\bigoplus_{m\in\bbZ^{r_1}\times \bbZ_+^{r_2}} L^2(M\bs K,m).\]
Note that every irreducible representation of $\prod_j (K_j/\{\pm I\})$ occurs in this decomposition exactly once.

For any $\phi\in L^2(M\bs K,m)$ (extended to a function on $G$ by $\phi(nak)=\phi(k)$) we attach the Eisenstein series
\[E(\phi,s,g)=\sum_{\gamma\in \Gamma_\infty\bs \Gamma} \vphi_{s}(\gamma g)\phi(\gamma g),\]
and define the constant term $E^o(\phi,s,g)$ in the same way. We use the raising and lowering operators to obtain the analogue of \eqref{e:constant}.
\begin{prop}\label{p:const}
For any $\phi\in L^2(M\bs K,m)$
\[E^o(\phi,s,g)=\bigg(\vphi_{s}(g)+P_m(s)\calC(s)\vphi_{1-s}(g)\bigg)\phi(g),\]
where $\calC(s)$ is as in \eqref{e:constant} and
\begin{equation}\label{e:Pm}
P_{m}(s)=\prod_{j=1}^n \prod_{k=0}^{|m_j|-1} \frac{\mu_j(1-s)+k}{\mu_js+k}.
\end{equation}
\end{prop}
\begin{proof}
Since $L^2(M\bs K,m)=\bigotimes L^2(M_j\bs K_j,m_j)$ it is enough to prove this for functions of the form
$\phi(k)=\prod_j \phi_j(k_j)$ where each $\phi_j\in L^2(M_j\bs K_j,m_j)$.

We first show this for the specific case when $\phi_m=\prod_j\phi_{m_j}$ where each $\phi_{m_j}$ is of $K_j$-weight $m_j$ for $j\leq r_1$ and maximal weight $m_j$ for $j>r_1$. For simplicity, we assume that $m_j\geq 0$ for $j\leq r_1$ (otherwise the argument is the same with the lowering operator instead of the raising operator). When $m=0$ this is \eqref{e:constant}. Next, applying the raising operators $a_j^+$ (see \eqref{e:raising1} and \eqref{e:raising2}) we get
\[a_j^+ \vphi_s(g)\phi_m(k)=c_m(\mu_js+m_j)\vphi_s(g)\phi_{m+e_j}(k).\]
Since the action of $a_j^+$ commutes with the left action of $G$, it commutes with the $\Gamma$ action, implying that
\[a^+E(\phi_m,s,g)=c_m(\mu_js+m_j)E(\phi_{m+e_j},s,g),\]
and it commutes with the action of $N$ so that
\[a^+E^o(\phi_m,s,g)=c_m(\mu_js+m_j)E^o(\phi_{m+e_j},s,g).\]
Now, by induction, we have that $E^o(\phi_m,s,g)=\left(\vphi_s(g)+\calC(s)P_m(s)\vphi_{1-s}(g)\right)\phi_m(g)$ so that
\[a^+E^o(\phi_m,s,g)=c_m\left((\mu_js+m_j)\vphi_s(g)+\calC(s)P_m(s)(m_j+\mu_j(1-s))\vphi_{1-s}(g)\right)\phi_{m+e_j}(g).\]
Comparing the two we get
\[E^o(\phi_{m+e_j},s,g)=\vphi_s(g)\phi_{m+e_j}(g)+\calC(s)P_{m+e_j}(s)\vphi_{1-s}(g)\phi_{m+e_j}(g).\]

This proves the result for $\phi=\prod_j\phi_{m_j}$ with each $\phi_{m_j}$ of maximal weight. Next, applying the lowering operators $a_{K_j}^-$ (for $j>r_1$) we get that the same formula is satisfied by any $\phi=\prod_j \phi_j$ with $\phi_j\in L^2(M_j\bs K_j,m_j)$ of arbitrary $M$-weights.
\end{proof}

\subsection{Explicit formula}
We are now in a position to give upper and lower bounds for $\norm{\Theta_f}^2$ that depend explicitly on the poles of the constant term $\calC(s)$ (cf. \cite[Page 108]{Harishchandra68}).

For each one of the spaces $L^2(M\bs K,m)$ we fix an orthonormal basis
$$\{\phi_{m,l}|\;l\in \bbZ^{r_2},|l_j|\leq |m_j|\}.$$
For any $f\in C^\infty(Q\bs G)$ let
\begin{equation}\label{e:fml}
\hat{f}_{m,l}(a)=\int_K f(ak)\overline{\phi_{m,l}(k)}dk,
\end{equation}
and define the following function
\begin{equation}\label{e:Mf}
M_f(s)=\sum_{m,l}P_m(s)\left|\int_\bbR \hat{f}_{m,l}(a_{\eta t})e^{-st}dt\right|^2,
\end{equation}
with $P_m(s)$ as in \eqref{e:Pm}. We then have
\begin{prop}\label{p:ThetaMain}
Let $\{s_1,\ldots,s_p\}\subseteq(\tfrac{1}{2},1)$ denote the exceptional poles of $\calC_\Gamma(s)$ (if they exist) and let $c_j$ denote the residue at $s_j$ and $c_0=\frac{|\omega_\Gamma|}{v_\Gamma}$ the residue at $s_0=1$. Then
 \begin{enumerate}
 \item For all $f\in C^\infty_c(Q\bs G)$ we have the upper bound
\[\norm{\Theta_f}^2\leq 2c_0\norm{f}_2^2+
c_0^2\norm{f}_1^2+c_0\sum_{j=1}^{p}c_jM_f(s_j).\]
\item For all positive $f\in C^\infty_c(Q\bs G)$ we have the lower bound
\[\norm{\Theta_f}^2\geq c_0^2\norm{f}_1^2+c_0\sum_{j=1}^{p}c_jM_f(s_j).\]
\end{enumerate}
\end{prop}

We will postpone the proof of Proposition \ref{p:ThetaMain} to the end of this section. We now show how it implies Theorem \ref{t:thetabound} and Theorem \ref{t:ThetaR}.

\subsection{Proof of Theorem \ref{t:thetabound}}
First recall that for $\Gamma=\SL_2(\calO_K)$ there are no poles in $(\tfrac{1}{2},1)$ (see remark \ref{r:CongruenceConstant}).
In particular for these lattices Proposition \ref{p:ThetaMain} directly implies
\begin{cor}\label{c:congruence}
For $\Gamma=\SL_2(\calO_K)$ any $f\in C^\infty_c(Q\bs G)$ satisfies
\[\norm{\Theta_f}^2 \leq 2\frac{|\omega_{\Gamma}|}{v_\Gamma}\norm{f}_2^2+
(\frac{|\omega_{\Gamma}|}{v_{\Gamma}})^2\norm{f}_1^2.\]
\end{cor}
We note that (for $n=1$) there are non arithmetic lattices, and there are also arithmetic (non-congruence) lattices for which $\calC_\Gamma(s)$ has nontrivial poles in $(\tfrac{1}{2},1)$. In fact, as noticed by Selberg \cite{Selberg65}, there are arithmetic lattices with poles arbitrarily close to $1$. For these we need to control the contribution of the terms $M_f(s_j)$.
That is, we need a bound of the form
\begin{equation}\label{e:hypothesis}
\forall s\in(\tfrac{1}{2},1),\quad M_f(s)\lesssim_s \norm{f}_2^2+\norm{f}_1^2.
\end{equation}

Though we suspect that such a bound should hold in general, we were not able to prove it by analyzing the terms $M_f(s)$ directly. Instead, we can get the desired estimate by comparing the norms of theta functions corresponding to different lattices.
To do this we need the following simple lemma.
\begin{lem}\label{l:SG}
Let $\Lambda\subseteq \Gamma$ be a subgroup of finite index. Then any positive $f\in C^\infty_c(Q\bs G)$ satisfies
\[\norm{\Theta_f^{\Lambda}}^2\leq \frac{[\Gamma_\infty:\Lambda_\infty]^2}{[\Gamma:\Lambda]} \norm{\Theta_f^{\Gamma}}^2,\]
where the norms are taken in $L^2(\Lambda\bs G,\sigma_{\Lambda})$ and $L^2(\Gamma\bs G,\sigma_\Gamma)$ respectively.
\end{lem}
\begin{proof}
From positivity of $f$ we have,
\[\Theta_f^{\Lambda}(g)\leq \sum_{\Lambda_{\infty}\bs\Gamma} f(\gamma g)=[\Gamma_{\infty}:\Lambda_\infty] \Theta_f^{\Gamma}(g).\]
Plugging this positivity bound in the identity  \eqref{e:unfold2}
we get
\begin{eqnarray*}
\norm{\Theta_f^\Lambda}^2&=& \int_{\Lambda_\infty\bs G}  \overline{f(g)} \Theta_f^{\Lambda}(g)d\sigma_{\Lambda}(g)\\
&\leq &[\Gamma_\infty:\Lambda_{\infty}]\int_{\Lambda_\infty\bs G}  \overline{f(g)} \Theta_f^{\Gamma}(g)d\sigma_{\Lambda}(g)\\
&=&\frac{[\Gamma_\infty:\Lambda_{\infty}]^2}{[\Gamma:\Lambda]}\int_{\Gamma_\infty\bs G}  \overline{f(g)} \Theta_f^{\Gamma}(g)d\sigma_{\Gamma}(g),
\end{eqnarray*}
where the last equality can be seen by writing the fundamental domain for $\Lambda_\infty\bs G$ is a union of $[\Gamma_{\infty}:\Lambda_\infty]$ translations of a fundamental domain for $\Gamma_\infty\bs G$ and noting that both $f$ and $\Theta_f^{\Gamma}$ are invariant under $\Gamma_{\infty}$. Using \eqref{e:unfold2} again, this time for $\Gamma$, concludes the proof.

\end{proof}


We can now prove the bound \eqref{e:hypothesis} for any value of $s\in (\tfrac{1}{2},1)$ that occurs as a pole for some arithmetic lattice.
\begin{prop}\label{p:Mfbnd}
Let $\tilde{\Gamma}\subseteq G$ denote an arithmetic lattice and let $s_1\in (\tfrac{1}{2},1)$ be a pole of $\calC_{\tilde{\Gamma}}(s)$.
We then have that for all positive $f\in C^\infty_c(Q\bs G)$
\[M_f(s_1)\lesssim_{s_1,\tilde{\Gamma}} \norm{f}_2^2+\norm{f}_1^2.\]
\end{prop}
\begin{proof}
The condition that $\tilde{\Gamma}$ is arithmetic implies that it is commensurable to $\Gamma=\SL_2(\calO_K)$. Let $\Lambda=\Gamma\cap \tilde{\Gamma}$, then $\Lambda$ is of finite index in both. In particular, if $s_j\in (\tfrac{1}{2},1)$ is a pole of $\calC_{\tilde{\Gamma}}(s)$ then it is also a pole of $\calC_{\Lambda}(s)$ with a positive residue $c_j>0$. We then have, from the second part of Proposition \ref{p:ThetaMain} together with Lemma \ref{l:SG} and Corollary \ref{c:congruence}, that for any positive $f\in C^\infty(Q\bs G)$
\[c_jM_f(s_j)\leq \norm{\Theta_f^\Lambda}^2\leq [\Gamma_\infty:\Lambda_\infty]\norm{\Theta_f^{\Gamma}}^2\lesssim \norm{f}_1^2+\norm{f}_2^2.\]
\end{proof}
Theorem \ref{t:thetabound} now follows directly from Propositions \ref{p:Mfbnd} and the first part of Proposition \ref{p:ThetaMain}.

\subsection{Proof of Theorem \ref{t:ThetaR} (the $\SL_2(\bbR)$ case)}
Here we allow $\Gamma\subseteq G=\SL_2(\bbR)$ to be an arbitrary non-uniform lattice and we consider the following family of functions:
For any large $\lambda>1$ and small $\epsilon>0$ we let $f^{(\lambda)}\in C^\infty_c(Q\bs G)$ be given by
\[f^{(\lambda)}(n_xa_{t}k_\theta)=v_\lambda(t)\psi_\lambda(\theta),\]
where $v_\lambda(t)$ is supported on $[-(1+\epsilon)\log(\lambda),0]$ takes values in $[0,1]$ and equals $1$ on the interval $[-(1+\epsilon)\log(\lambda)+1,-1]$ and $\psi_\lambda(x)=\psi(\lambda x)$ with $\psi\in C^\infty(\bbR)$ compactly supported and takes values in $[0,1]$.

Note that for these functions $\norm{f^{(\lambda)}}_2^2\asymp \norm{f^{(\lambda)}}_1\asymp \lambda^{\epsilon}$.
Hence, by Proposition \ref{p:ThetaMain}, it is enough to show that
\[M_{f^{(\lambda)}}(s)\lesssim_s \lambda^{2\epsilon} \mbox{ for all } s\in(\tfrac{1}{2},1).\]
For our specific family of functions we also have
\[M_{f^{(\lambda)}}(s)=\left(\sum_m P_m(s)|\hat{\psi}_\lambda(2m)|^2\right)|\int_\bbR v_\lambda(t)e^{-st}dt|^2\asymp
\lambda^{2s+2\epsilon-2}\left(\sum_m P_m(s)|\hat{\psi}(\frac{2m}{\lambda})|^2\right).\]
We can estimate for $s\in (\tfrac{1}{2},1)$
\[P_m(s)=\prod_{k=0}^{m-1}\frac{1-s+k}{s+k}\asymp_s \frac{1}{(m+1)^{2s-1}}.\]
Indeed
\begin{eqnarray*}\log(P_m(s))&=&\log(\frac{s-1}{s})+\sum_{k=1}^{m-1}(\log(1+\tfrac{1-s}{k})-\log(1+\tfrac{s}{k}))\\
&=&(1-2s)\sum_{k=1}^{m-1}\frac{1}{k}+O_s(1)=(1-2s)\log(m+1)+O_s(1)
\end{eqnarray*}
We thus get that
\[M_{f^{(\lambda)}}(s)\asymp_s \lambda^{2(s-1)+2\epsilon}\sum_{m=1}^\infty m^{1-2s}|\hat{\psi}(\frac{2m}{\lambda})|^2.\]
For any $m\leq 2\lambda$ estimate $|\hat{\psi}(\frac{2m}{\lambda})|\lesssim 1$ while for $m\in[2\lambda k, 2\lambda(k+1)]$ we can estimate $|\hat{\psi}(\frac{2m}{\lambda})|\lesssim \frac{1}{k}$. We thus get that
\begin{eqnarray*}
\sum_m m^{1-2s}|\hat{\psi}(\frac{2m}{\lambda})|^2&=& \sum_{k=0}^\infty \sum_{m=2\lambda k+1}^{2\lambda(k+1)}m^{1-2s}|\hat{\psi}(\frac{2m}{\lambda})|^2\\
&\lesssim &  \sum_{m=1}^{2\lambda} m^{1-2s}+ \lambda^{2-2s}\sum_{k=1}^\infty k^{1-2s-2} \lesssim  \lambda^{2(1-s)},
\end{eqnarray*}
so, indeed,  $M_{f^{(\lambda)}}(s)\lesssim_s \lambda^{2\epsilon}$ for any $s\in (\tfrac{1}{2},1)$.
\begin{rem}
We note that this bound is optimal. Indeed, if we assume that both $\psi$ and $\hat{\psi}$ are positive
then the same argument also gives a lower bound $M_{f^{(\lambda)}}(s)\gtrsim_s \lambda^{2\epsilon}$. Moreover, for the end point $s=1$ we get
$M_{f^{(\lambda)}}(1)\gtrsim \lambda^{2\epsilon}\log(\lambda)$ so the condition that $s<1$ is crucial.
\end{rem}

\subsection{Proof of Theorem \ref{t:ThetaR} (the $\SL_2(\bbC)$ case)}
Let $\Gamma\subseteq G=\SL_2(\bbC)$ be an arbitrary non-uniform lattice and consider the following family of functions:
For any large $\lambda>1$ and small $\epsilon>0$ we let $f^{(\lambda)}\in C^\infty_c(Q\bs G)$ be given by
\[f^{(\lambda)}(n_xa_{t/2}k)=v_\lambda(t)\phi^{(\lambda)}(k),\]
where $v_\lambda(t)$ is supported on $[-(3+\epsilon)\log(\lambda),0]$ takes values in $[0,1]$ and equals $1$ on the interval $[-(3+\epsilon)\log(\lambda)+1,-1]$, and $\phi^{(\lambda)}\in C^\infty(M\bs K)$ is given by
\[\phi^{(\lambda)}(k_{\theta,\alpha,\beta})=\psi_\lambda(\sin(\theta),(\alpha-\beta)),\]
where $\psi$ is a smooth compactly supported function on $\bbR^2$ and $\psi_\lambda(x_1,x_2)=\psi(\lambda x_1,\lambda x_2)$. (Note that any smooth function on $M\bs K$ can be written as $\phi(k_{\theta,\alpha,\beta})=\psi(\sin(\theta),\alpha-\beta)$.)

As above, in order to prove the theorem it is enough to show that
$M_{f^{(\lambda)}}(s)\lesssim_s \norm{f^{(\lambda)}}_1^2\asymp \lambda^{2\epsilon}$ for all $s\in (\tfrac{1}{2},1)$.
Here we were not able to prove it directly as before.
Instead, we will show that for $\lambda$ sufficiently large $M_{f^{(\lambda)}}(s)$ is an increasing function of $s$.
The result will then follow from the bounds for arithmetic lattices and the fact that there are arithmetic lattices with poles arbitrary close to one.

We first need a few preliminary estimates.
\begin{lem}
If $f\in C^\infty(Q\bs G)$ factors as $f(a_{\eta t}k)=v(t)\phi(k)$ then $M_f(s)$ also factors as
\[M_f(s)=\left(\sum_{m=0}^\infty P_m(s)\norm{\phi_m}_2^2\right)|\int_\bbR v(t)e^{-st}dt|^2,\]
 where $\phi_m$ denotes the projection of $\phi$ to $L^2(M\bs K,m)$.
\end{lem}
\begin{proof}
We can write the projection $\phi_m$ as
\[\phi_m(k)=\sum_{l=-m}^m c_{m,l}\phi_{m,l}(k),\]
with $c_{m,l}=\int_K \phi(k)\overline{\phi_{m,l}(k)}dk$ so that
\[\norm{\phi_m}^2=\sum_{l=-m}^m |c_{m,l}|^2.\]
The factorization of $f$ implies that
$$\hat{f}_{m,l}(a_{\eta t})=v(t)\int_K \phi(k)\overline{\phi_{m,l}(k)}dk=c_{m,l}v(t),$$
so that
\[M_{f}(s)=\sum_{m}P_m(s)\sum_{l=-m}^m | c_{m,l}|^2\left|\int_\bbR v(t)e^{-t}dt\right|^2=\left(\sum_{m} P_m(s)\norm{\phi_m}^2\right)\left|\int_\bbR v(t)e^{-t}dt\right|^2.\]
\end{proof}
As before we can estimate the function
\[P_m(s)=\prod_{k=0}^{m-1} \frac{2(1-s)+k}{2s+k}\asymp_s (m+1)^{2-4s},\]
and the integral $\int_\bbR v_\lambda(t)e^{-st}dt\asymp_s \lambda^{s(3+\epsilon)}$ to get that
\[M_{f^{(\lambda)}}(s)\asymp_s \tilde{M}(\lambda,s)= \lambda^{2s(3+\epsilon)}\sum_{m=0}^\infty \frac{||\phi_m^{(\lambda)}||^2}{(m+1)^{4s-2}}.\]
In particular, we have $M_{f^{(\lambda)}}(s)\lesssim_s ||f^{(\lambda)}||_1^2$ if and only if  $\tilde{M}(\lambda,s)\lesssim_s ||f^{(\lambda)}||_1^2$.

\begin{lem}\label{l:monotone}
There is a constant $\lambda_0$ such that for any $\lambda>\lambda_0$ if $\tilde{M}(\lambda,s_0)\geq \lambda^{3\epsilon/2}$, then $\tilde{M}(\lambda,s)>\tilde{M}(\lambda,s_0)$ for all $s\in (s_0,1)$.
\end{lem}
\begin{proof}
We will show that for all $s\in [s_0,1)$ the derivative
\[\pd{\tilde{M}}{s}(\lambda,s)=2(3+\epsilon)\log(\lambda)\lambda^{2s(3+\epsilon)}\sum_{m=0}^\infty \frac{||\phi^{(\lambda)}_m||^2}{(m+1)^{4s-2}}-4\lambda^{2s(3+\epsilon)}\sum_{m=0}^\infty \frac{\log(m+1)||\phi^{(\lambda)}_m||^2}{(m+1)^{4s-2}}\geq 0.\]
To do this, for any $L>2$ we can bound the sum in the second term by
\[\sum_{m=0}^\infty \frac{\log(m+1)||\phi^{(\lambda)}_m||^2}{(m+1)^{4s-2}}\leq \log(L+1)\sum_{m=0}^{L} \frac{||\phi^{(\lambda)}_m||^2}{(m+1)^{4s-2}}+\frac{\log(L+1)}{(L+1)^{4s-2}}\sum_{m>L}||\phi^{(\lambda)}_m||^2.\]
and using Parseval's identity we can bound the second sum by
\[\frac{\log(L+1)}{(L+1)^{4s-2}}\sum_{m>L}||\phi^{(\lambda)}_m||^2\leq \frac{\log(L+1)}{(L+1)^{4s-2}}\int_K\phi^{(\lambda)}(k)dk\ll \frac{\log(L+1)}{(L+1)^{4s-2}\lambda^3}. \]
We thus get that
\[\sum_{m=0}^\infty \frac{\log(m+1)||\phi^{(\lambda)}_m||^2}{(m+1)^{4s-2}}\leq \log(L+1)\left(\sum_{m=0}^{\infty} \frac{||\phi^{(\lambda)}_m||^2}{(m+1)^{4s-2}}+O(\frac{\log(L+1)}{ L^{4s-2}\lambda^3})\right).\]
Taking $L=\lambda^{\frac{3}{2}+\frac{\epsilon}{3}}-1$ and recalling that $s<1$ we get that 
\[\pd{}{s}\tilde{M}(\lambda,s)\geq \epsilon\log(\lambda)M(\lambda,s)+O(\log(\lambda)\lambda^{4\epsilon/3}).\]
Hence, there is a constant $C>0$ (independent on $s$ or on $\lambda$) such that
$$\pd{}{s}\tilde{M}(\lambda,s)\geq \log(\lambda)(\epsilon M(\lambda,s)-C\lambda^{4\epsilon/3}).$$
Let $\lambda_0$ be large enough so that $\epsilon\lambda_0^{\epsilon/6}>4C$, then, for all $\lambda>\lambda_0$
if $\tilde{M}(\lambda,s_0)\geq \lambda^{3\epsilon/2}$ then $\pd{}{s}\tilde{M}(\lambda,s_0)>\log(\lambda)\lambda^{3\epsilon/2}\epsilon/2>0$ and hence $\tilde{M}(\lambda,s)$ is increasing at $s_0$.
Consequently, $\tilde{M}(\lambda,s)\geq \lambda^{3\epsilon/2} $ and is also increasing for all $s_0\leq s<1$.
\end{proof}

We can now conclude the proof of Theorem \ref{t:ThetaR}.
As mentioned above, it is enough to show that $\tilde{M}(\lambda,s)\lesssim_s \norm{f^{(\lambda)}}_1^2$ for all $s\in(\tfrac{1}{2},1)$.
Assume that there is some $s_0\in (\tfrac{1}{2},1)$ for which this is false. That is, there is a sequence $\lambda_\ell\to\infty$ for which
$\tilde{M}(\lambda_\ell,s_0)/||f^{(\lambda_\ell)}||_1^2\to \infty$. In particular, we have that $\tilde{M}(\lambda_\ell,s_0)\geq \lambda^{3\epsilon/2}$  for all $\ell$ sufficiently large (recall that $\norm{f^{(\lambda)}}_1\asymp \lambda^\epsilon$).
Consequently, Lemma \ref{l:monotone} tells us that $\tilde{M}(\lambda_\ell,s)/||f^{(\lambda_\ell)}||_1^2\to\infty$ for all $s\in(s_0,1)$.
On the other hand, we can find an arithmetic lattice $\Gamma^*\subseteq \SL_2(\bbZ[i])$ such that $\calC_{\Gamma^*}(s)$ has a pole $s_1\in (s_0,1)$ (see e.g.,  \cite{Selberg65}).
For this pole, Proposition \ref{p:Mfbnd} tells us that $\tilde{M}(\lambda, s_1)\lesssim M_{f^{(\lambda)}}(s_1)\lesssim \norm{f^{(\lambda)}}_1^2$
in contradiction.

\subsection{Proof of Proposition \ref{p:ThetaMain}}
We now go back to complete the proof of Proposition \ref{p:ThetaMain}.
Note that for any $f\in C^\infty_c(Q\bs G)$ we have $f=\sum_{m,l}f_{m,l}$ where $f_{m,l}(ak)=\hat{f}_{m,l}(a)\phi_{m,l}(k)$ and that from orthogonality
$\norm{\Theta_f}^2=\sum_{m,l}\norm{\Theta_{f_{m,l}}}^2.$
We can thus reduce the problem to the case where $f=f_{m,l}$ for some fixed $m,l$.
\begin{prop}\label{p:norm}
Let $f\in C^\infty_c(Q\bs G)$ be of the form $f(a_{\eta t}k)=v(t)\phi(k)$
where $v\in C^\infty_c(\bbR)$ and $\phi\in L^2(M\bs K,m)$ for some fixed $m$.
Let $\tfrac{1}{2}<s_p<\ldots<s_1<s_0=1$ denote the poles of $\calC(s)$ and let $c_j=\Res_{s=s_j}\calC(s)$. We then have
\[c_0\sum_{j=0}^p c_j P_m(s_j)\left|\int_\bbR v(t)e^{-s_jt}dt\right|^2\leq \norm{\Theta_f}^2\leq c_0\big(2\norm{f}_2^2+\sum_{j=0}^p c_j P_m(s_j)\left|\int_\bbR v(t)e^{-s_jt}dt\right|^2\big).\]
\end{prop}
\begin{proof}
Let $\hat{v}(r)=\tfrac{1}{\sqrt{2\pi}}\int_\bbR v(t)e^{-irt}dt$ denote the Fourier transform of $v$; for any $\sigma\in \bbR$ we have
$$v(t)=\tfrac{1}{\sqrt{2\pi}}\int_\bbR \hat{v}(r-i\sigma)e^{it(r-i\sigma)}dr=\tfrac{1}{\sqrt{2\pi}}\int_\bbR \hat{v}(r-i\sigma)\vphi_{\sigma+ir}(a_{\eta t})dr.$$
Consequently we can write,
\[f(g)=\tfrac{1}{\sqrt{2\pi}}\int_\bbR \hat{v}(r-i\sigma)\vphi_{\sigma+ir}(g)\phi(g)dr,\]
and summing over $\Gamma_\infty\bs \Gamma$ we get
\begin{eqnarray*}
\Theta_f(g)&=& \frac{1}{\sqrt{2\pi}}\int_{\bbR}\hat{v}(r-i\sigma)\sum_{\gamma\in \Gamma_\infty\bs \Gamma}\vphi_{\sigma+ir}(\gamma g)\phi(\gamma g)dr\\
&=& \frac{1}{\sqrt{2\pi}}\int_{\bbR}\hat{v}(r-i\sigma)E(\phi,\sigma+ir,g)dr.
\end{eqnarray*}
Integrating this over $\calF_{\calO_\Gamma}$ gives
\begin{eqnarray*}
\lefteqn{\int_{\calF_{\calO_\Gamma}} \Theta_f(n_{\bf{x}}ak)d{\bf{x}} =\frac{1}{\sqrt{2\pi}}\int_{\bbR}\hat{v}(r-i\sigma)\int_{\calF_{\calO_\Gamma}}E(\phi,\sigma+ir,n_{\bf{x}}ak)d{\bf{x}}dr}\\
&&=\frac{|\calF_{\calO_\Gamma}|}{\sqrt{2\pi}}\int_{\bbR}\hat{v}(r-i\sigma)E^o(\phi,\sigma+ir,ak)dr,
\end{eqnarray*}
and using Proposition \ref{p:const} for $E^o(\phi,s,g)$ we get
\begin{eqnarray*}
\int_{\calF_{\calO_\Gamma}} \Theta_f(n_{\bf{x}}ak)d{\bf{x}} &=&\frac{|\calF_{\calO_\Gamma}|\phi(k)}{\sqrt{2\pi}}\bigg(\int_{\bbR}\hat{v}(r-i\sigma)\vphi_{\sigma+ir}(a)dr\\
&+& \int_{\bbR}\hat{v}(r-i\sigma) \calC(\sigma+ir)P_m(\sigma+ir)\vphi_{1-\sigma-ir}(a)dr\bigg).\\
\end{eqnarray*}
Now shift the contour of integration to the line $\sigma=\tfrac{1}{2}$ (picking up possible poles) to get
\begin{eqnarray}\label{e:contour}
\int_{\calF_{\calO_\Gamma}} \Theta_f(n_{\bf{x}}ak)d{\bf{x}} &=&\frac{|\calF_{\calO_\Gamma}|\phi(k)}{\sqrt{2\pi}}\bigg(\int_{\bbR}\hat{v}(r-\tfrac{i}{2})\vphi_{\tfrac{1}{2}+ir}(a)dr\\
\nonumber &+& \int_{\bbR}\hat{v}(r-\tfrac{i}{2}) \calC(\tfrac{1}{2}+ir)P_m(\tfrac{1}{2}+ir)\vphi_{\tfrac{1}{2}-ir}(a)dr\\
\nonumber&+& 2\pi\sum_j c_j P_m(s_j)\hat{v}(-is_j)\vphi_{1-s_j}(a)\bigg).
\end{eqnarray}
We recall the formula \eqref{e:unfold2}, that when written in the coordinates at the cusp reads
\[\norm{\Theta_f}^2=\frac{R_\Gamma}{v_{\Gamma}}\int_K\int_{[-1,1]^{n-1}\times\bbR} \overline{f(a_{D\bf{t}}k)}e^{-t_n}\int_{\calF_{\calO_\Gamma}} \Theta_f(n_{\bf{x}}a_{D\bf{t}}k)d{\bf{x}}d{\bf{t}}dk.\]

Plugging \eqref{e:contour} in this formula, noting that $f(a_{D\bf{t}}k)=v(t_n)\phi(k)$, that $\vphi_s(a_{D\bf{t}})=e^{st_n}$ and recalling that $c_0=\frac{R_{\Gamma}2^{n-1}|\calF_{\calO_\Gamma}|}{v_{\Gamma}}$ (see \eqref{e:residue}) we get
\begin{eqnarray*}\norm{\Theta_f}^2 &=&c_0\bigg( \int_\bbR |\hat{v}(r-\tfrac{i}{2})|^2dr
 +\int_\bbR \hat{v}(r-\tfrac{i}{2})\overline{\hat{v}(-r-\tfrac{i}{2})}\calC(\tfrac{1}{2}+ir)P_m(\tfrac{1}{2}+ir)dr\\
& +& 2\pi\sum_j c_j P_m(s_j)|\hat{v}(-is_j)|^2\bigg).
\end{eqnarray*}
Now, for the first term, by Plancherel, we have
\[\int_\bbR |\hat{v}(r-\tfrac{i}{2})|^2dr=\int_\bbR |v(t)|^2e^{-t}dt=\norm{f}_2^2.\]
Using Cauchy-Schwartz and the fact that $|\calC(\tfrac{1}{2}+ir)|\leq 1$ and $|P_{m}(\tfrac{1}{2}+ir)|=1$, we see that the absolute value of the second term is bounded by the first term.
For the last term we have for each pole $2\pi|\hat{v}(-is_j)|^2=\left|\int_\bbR v(t)e^{-s_jt}dt\right|^2$,
implying the upper and lower bounds for $\norm{\Theta_f}^2$.
%
\end{proof}

We can now conclude the proof of Proposition \ref{p:ThetaMain}. For any $f\in C^\infty_c(Q\bs G)$ let $f_{m,l}(ak)=\hat{f}_{m,l}(a)\phi_{m,l}(k)$ with $\hat{f}_{m,l}$ defined in \eqref{e:fml}. From orthogonality we get that
$$\norm{\Theta_f}^2=\sum_{m,l}\norm{\Theta_{f_{m,l}}}^2,$$
and we can use Proposition \ref{p:norm} to estimate each of the terms $\norm{\Theta_{f_{m,l}}}^2$ separately and sum all the contributions.

First, for the $L^2$-norms we have $2\sum_{m,l}\norm{f_{m,l}}_2^2=2\norm{f}_2^2$.
Next, the contribution of the exceptional poles $\tfrac{1}{2}<s_j<1$ (if they exist) is
$\sum_jc_jM_f(s_j)$. Finally, since $P_m(1)=0$ unless $m=0$, the pole at $s_0=1$ only contributes for $m=l=0$  and its contribution is precisely
\[c_0\left|\int_\bbR \hat{f}_{0,0}(a_{\eta t})e^{-t}dt\right|^2=c_0\left|\int_{Q\bs G}f(g)dg\right|^2\leq c_0\norm{f}_1^2.\]
If we further assume that $f$ is positive, then the last inequality is an equality implying the lower bound.


\section{Logarithm laws}~\label{sec;vol}
Let $X=\Gamma\bs G$ be as above and let $\{u_s\}_{s\in \bbR}\subset G$ denote a one parameter unipotent subgroup.
The goal of this section is to prove Theorem \ref{t:logarithmlaw}, that is, to show that
\[\lims_{s\to\infty}\frac{d(o,x u_s)}{\log(s)}=1 \mbox{ for $\sigma$-a.e. }x\in X.\]

Note that for any conjugated and rescaled flow of the form $\tilde{u}_s=ku_{\lambda s} k^{-1}$ with $k\in K$ and $\lambda>0$ we have that
$$\lims_{s\to\infty}\frac{\d(o,x\tilde{u}_s)}{\log(s)}
=\lims_{s\to\infty}\frac{\d(o,\tilde{x}u_s)}{\log(s)}$$
where $\tilde{x}=xk$. Consequently, we may assume without loss of generality that $\Gamma$ has a cusp at infinity and that the unipotent flow is given by
 \begin{equation}\label{e;unipotent}
 u_s=n_{s\mathbf{y}}^-=(n^{-}_{sy_1},\ldots,n^{-}_{sy_n}) \end{equation}
where $n_x^-=\begin{pmatrix} 1 & 0\\ x & 1\end{pmatrix}$, and $\ybf\in [0,1]^n$ is fixed and satisfies  $\max\{y_j|1\leq j\leq n\}=1$.
Throughout the rest of this section we will fix such a $\ybf$ and a unipotent flow $u_t$ as above.
\subsection{Translates of cusp neighborhoods}\label{s:construct}
We now give the precise constructions of the sets $\frakD_k\subseteq Q\bs G$ and $Y_{\Dfrak_k}\subset \Gamma\bs G$ described in the introduction. For any set $\Dfrak\subset Q\bs G$ denote by $|\Dfrak|$ its measure with respect to the Haar measure on $Q\bs G$. We define the set $Y_{\Dfrak}\subset \Gamma\bs G$ corresponding to $\Dfrak$ by
\begin{equation}\label{e:YD}
 Y_{\Dfrak}=\{\Gamma g\in \Gamma\bs G| Q\gamma g\in \Dfrak \mbox{ for some } \gamma\in \Gamma\}.
\end{equation}

 We will work in slightly greater generality, and fix an arbitrary increasing sequence of real numbers $r_{\ell}\to \infty$ satisfying that $\sum_{\ell} e^{-r_\ell}=\infty$.
For any $k\in \bbN$ let $p(k)\in \bbN$ such that
$$\lim_{k\to\infty}\sum_{\ell=k}^{p(k)}e^{-r_{\ell}}=\infty.$$
Eventually, we will take $r_\ell=(1-\epsilon)\log(\ell)$ in which case we can take $p(k)=2k$.
\begin{defn}\label{def:Dk}
Let $A^1(\tau)=\{a_{\eta t}|t\geq \tau\}$ and define the sets
$$\Dfrak_k=Q\bs \bigcup_{\ell=k}^{p(k)} Q A^1(r_\ell)Ku_{-\ell}\subseteq Q\bs G$$
\end{defn}

\begin{lem}\label{l:Yk}
For any $x\in Y_{\Dfrak_k}$ there is $\ell \geq k$ such that
$$\d(o,xu_\ell)\geq r_\ell+O(1).$$
\end{lem}
\begin{proof}
From the construction, any $x\in Y_{\Dfrak_k}$ can be written as
$x=\Gamma gu_{-\ell}$ for some $\ell\geq k$ and $g\in Q A^1(r_\ell)K$.
Moreover, replacing $g$ if necessary with $\gamma g$ for a suitable $\gamma\in \Gamma_\infty$ we can take $g$ from the Siegel set $\omega A^1(r_\ell)K$. Now write $g=qa_{\eta t}k$ with $q\in \omega,\;k\in K$ and $t\geq r_\ell$, then from \eqref{e:cuspdist} and \eqref{e:distcusp} we have that indeed
\[\d(o,xu_{\ell})=\d_G(o,a_{\eta t})+O(1)=t+O(1)\geq r_\ell+O(1).\]
\end{proof}

\begin{lem}\label{l:vol}
$\lim_{k\to\infty} |\Dfrak_k|=\infty$.
\end{lem}
\begin{proof}
Let $N^{-}\subseteq G$ denote the group of lower triangular matrices. We note that $NMAN^-=\{\smatrix{a}{b}{c}{d}\in G| d\neq 0\}$ is a Zariski open dense subset of $G$ containing the identity. Thus a Zariski open dense subset of $Q\ba G$ has representative of the form $Qg=Qa_{\eta t} \nxn,$ where $\nxn\in N^-$ and $a_{\eta t}\in A^1.$ We also note that the Haar measure on $Q\ba G$ in these coordinate is given by $e^{-t}d\hh t\hh d\hh\xbf$ (up to a normalizing constant).

Let $B^-=\{n^-_{\bf{x}}| \max_j|x_j|<1\}$ denote a fixed neighborhood of the identity in $N^{-}$.
A simple computation then shows that there exists an absolute constant $c>0$ such that for any $\tau>c$ we have that
\begin{equation}\label{e;affineinclusion}
QA^1(\tau-c)B^-\subseteq QA^1(\tau)K
\end{equation}

Let $u_s=n_{s\ybf}^-$ be as above.
Fix some $j_0$ with $y_{j_0}=1$ and for any $\xbf\in\bbR^{r_1}\times\bbC^{r_2}$ let $\ell_{\xbf}=[\Re(x_{j_0})]\in \bbZ$.
We define the set
\begin{equation}~\label{e:tildeD}\tilde{\Dfrak}_k=\left\{Qa_{\eta t}\nxn \in Q\bs QAN^-\h\bigg\arrowvert\h \begin{array}{l}
 k\leq \ell_{\xbf}\leq p(k)\\ |x_j+\ell_{\xbf}y_j|\leq 1  \\ t\geq  r_{\ell_{\xbf}}-c
\end{array}\right\}\end{equation}
Note that $a_{\eta t}\nxn u_{s}=a_{\eta t}n_{s\ybf+\xbf}^-$ and hence
\begin{equation}
\tilde{\Dfrak}_k\subseteq Q\bs  \bigcup_{\ell=k}^{p(k)} QA^1(\tau-c)B^-u_{-\ell} \subseteq Q\bs  \bigcup_{\ell=k}^{p(k)} QA^1(r_\ell)Ku_{-\ell}=\Dfrak_k.
\end{equation}

Finally, the volume of $\tilde{\Dfrak}_k$ can be explicitly computed and it satisfies
\[|\tilde{\Dfrak}_k|\asymp \int_{k}^{p(k)}\int_{r_{[x]}}^{\infty}e^{-t}dtdx\asymp \sum_{\ell=k}^{p(k)}e^{-r_\ell}.\]
So in particular $|\Dfrak_k|\rightarrow\infty$ as $k\to\infty$.
\end{proof}


Next we want to bound the measure of $Y_{\Dfrak_k}$ from below.
\begin{lem}\label{l:meas-mink}
Assume that $\Gamma$ is arithmetic. Then there is a constant $\kappa_\Gamma>0$ depending only on $\Gamma$ such that $\sigma(Y_\Dfrak)\geq \kappa_\Gamma$
for all $\Dfrak\subseteq Q\bs G$ with $|\Dfrak|> 1$ and $|\partial \Dfrak|=0$.
\end{lem}

\begin{proof}
Without loss of generality we may assume that $\Dfrak$ is relatively compact (otherwise replace it with a relatively compact subset satisfying the same assumptions).
Let $\id_\Dfrak\in L^2(Q\bs G)$ denote the characteristic function of $\Dfrak$ that we lift to a $Q$-left-invariant function on $G$ which we continue to denote by $\id_{\Dfrak}$. Let $f\in C_c^\infty(Q\bs G)$ take values in $[0,1]$ and approximate $\id_{\Dfrak}$ (from above) in $L^1$ sufficiently well so that
$|\Dfrak|\leq \norm{f}_1\leq 2|\Dfrak|$. Let $\Theta_{\id_{\Dfrak}}$ and $\Theta_{f}$ be the corresponding theta functions. Our normalization of the Haar measures give the following form of the Siegel's integral formula (see Lemma \ref{l:unfolding})
\begin{equation}~\label{e;L^1norm}\int_{\homog}\Theta_{f}(g) d\sigma(g)=\frac{|\omega_{\Gamma}|}{v_{\Gamma}}\int_{Q\ba G}f(g) dg.\end{equation}
Note that $\Theta_{\id_{\Dfrak}}$ is supported on $Y_\Dfrak$ and use Cauchy-Schwartz inequality to get
\begin{equation}\label{e;cauchy-sch}
(\frac{|\omega_{\Gamma}|}{v_{\Gamma}})^2|\Dfrak|^2=\left(\int_{\Gamma\bs G}\Theta_{\id_\Dfrak}(g)\hh d\hh \sigma(g)\right)^2=\left(\int_{Y_\Dfrak}\Theta_{\id_{\Dfrak}}\hh d\hh \sigma(g)\right)^2\leq \sigma(Y_\Dfrak)\norm{\Theta_{\id_\Dfrak}}^2.\end{equation}
Next, we can bound $\norm{\Theta_{\id_\Dfrak}}^2\leq\norm{\Theta_f}^2$ and from Theorem \ref{t:thetabound} we have that
\begin{equation*}
\norm{\Theta_{f}}^2\lesssim_{\Gamma} \norm{f}_2^2+\norm{f}_1^2.
\end{equation*}
Finally, bound $\norm{f}_2^2\leq \norm{f}_1\leq 2|\Dfrak|$ to get that there is a constant $\kappa>0$ depending only on $\Gamma$ such that 
$\sigma(Y_\Dfrak)\geq \kappa$.
\end{proof}
\begin{rem}
For $\Gamma=\SL_2(\calO_K)$ (or in any other case where $\calC_\Gamma(s)$ has no exceptional poles) we can get an explicit constant
$\sigma(Y_\Dfrak)\geq \frac{|\omega_\Gamma||\Dfrak|}{|\omega_\Gamma||\Dfrak|+v_\Gamma}$.
In particular, in these cases we have that  $\sigma(Y_{\Dfrak})\to 1$ as $|\Dfrak|\to \infty$.
\end{rem}

In the non arithmetic case, we use the same argument with the specific choice of functions $f^{(\lambda)}$ given in Theorem \ref{t:ThetaR}. Here we specialize to the case where the sequence $r_\ell=(1-\epsilon)\log(\ell)$ and observe that, writing the sets $\tilde{\Dfrak}_k$ defined in \eqref{e:tildeD} in polar coordinates one can show that for
the case $G=\SL_2(\bbR)$ we have
\[\set{Qa_{t}k_\theta| t\in [-(1+\epsilon)k,0],\; \tfrac{1}{2}\leq k\theta \leq 1}\subseteq \tilde{\Dfrak}_k,\]
and for the case $G=\SL_2(\bbC)$
\[\set{Qa_{\eta t}k_{\theta,\alpha,\beta}| t\in [-(3+\epsilon)k,0],\; \tfrac{1}{2}\leq k\sin(\theta)\leq 1,\; k|\alpha-\beta|\leq 1}\subset \tilde{\Dfrak}_k.\]
In both cases these sets can be approximated by $f^{(\lambda_k)}$ with $\lambda_k\asymp k$. We thus get the following
\begin{lem}\label{l:meas-mink2}
Let $G=\SL_2(\bbR)$ or $\SL_2(\bbC)$ and let $\Gamma\subset G$ denote a non uniform lattice.
For any $\epsilon>0$ let $r_\ell=(1-\epsilon)\log(\ell)$ and let
$\Dfrak_k=Q\bs \bigcup_{\ell=k}^{2k} Q A^1(r_\ell)Ku_{-\ell}\subseteq Q\bs G.$
Then there is a constant $\kappa>0$ such that $\sigma(Y_{\Dfrak_k})>\kappa$ for all $k\in \bbN$.
\end{lem}


\subsection{Proof of Theorem~\ref{t:logarithmlaw}}~\label{seclog}
We now complete the proof of Theorem ~\ref{t:logarithmlaw}.
For the sake of completeness we will prove both the upper and lower bound.

First for the upper bound. Fix $\e>0$ and let $r_{\ell}=(1+\e)\log(\ell)$. The sets
$$B_{r_\ell}=\{x|\d(o,x)\geq r_\ell\},$$ satisfy
\[\sum_{\ell=1}^\infty\sigma(B_{r_\ell}u_{-\ell})=\sum_{\ell=1}^\infty\sigma(B_{r_\ell})\asymp \sum_{\ell=1}^\infty\frac{1}{\ell^{1+\e}}<\infty.\]
Consequently, by Borel-Cantelli for $\sigma$-a.e. $x\in X$ we have that $\#\{\ell|xu_\ell\in B_{r_\ell}\}<\infty$ and hence
$$\lims_{\ell\rightarrow\infty}\frac{\d(o,x\hh u_\ell)}{\log \ell}\leq 1+\vare\h\h\mbox{for}\h\h\sigma\mbox{-a.e.}\h\h x\in X.$$
Since for all $x\in X$ all $s\in\bbR$ we have for $\ell=[s]$
$$|\d(o,x u_s)-\d(o,x u_{\ell})|\leq \d(xu_s,xu_{\ell})\leq \d_G(u_s,u_{\ell})=O(1),$$
we may replace the discrete limit over $\ell\in \bbN$ with a continuous limit over $s\in \bbR$.
Finally, since this holds for every $\e>0$ we get that
$$\lims_{s\rightarrow\infty}\frac{\d(o,x\hh u_s)}{\log s}\leq 1\h\h\mbox{for}\h\h\sigma\mbox{-a.e.}\h\h x\in X.$$

Next for the lower bound. Fix $\e>0$ and let $r_{\ell}=(1-\epsilon)\log(\ell)$.
Let $\Dfrak_k$ and $Y_{\Dfrak_k}$ be as above. We then have that, by Lemma \ref{l:vol}, $|\frakD_k|>1$ for $k$ sufficiently large and hence by Lemma \ref{l:meas-mink} (or Lemma \ref{l:meas-mink2} in the non-arithmetic case) there is some $\kappa>0$ such that $\sigma(Y_k)\geq \kappa>0$ for all $k$. Moreover, by lemma \ref{l:Yk},
for any $x\in Y_k$ there is some $\ell>k$ such that $\d(o, xu_\ell)\geq r_\ell$.

Let $Y=\cap_{\ell=1}^\infty \cup_{k=\ell}^\infty Y_k$.
Then $\sigma(Y)\geq \kappa$ and for every $x\in Y$ there is a sequence $\ell_k\to\infty$ such that $\d(o,xu_{\ell_k})\geq r_{\ell_k}$. Consequently, we have that
\[Y\subseteq \{x\in X|\lims_{s\to\infty}\frac{\d(o,xu_s)}{\log(s)}>1-\e\}.\]
But the latter set is invariant under the action of the flow $\{u_s\}_{s\in \bbR}$ and hence must have full measure.
Consequently, for $\sigma$-a.e. $x\in X$ we have that $\lims_{\ell\rightarrow\infty}\frac{\d(o,x\hh u_\ell)}{\log \ell}\geq 1-\vare$ and since this is true for any $\e>0$ we get that indeed
$$\lims_{s\rightarrow\infty}\frac{\d(o,x\hh u_s)}{\log s}= 1\h\h\mbox{for}\h\h\sigma\mbox{-a.e.}\h\h x\in X.$$

\begin{rem}
We remark that, in the arithmetic setting, we can repeat the same arguments with any increasing sequence of real numbers $\{r_{\ell}\}_{\ell\in\bbN}$ (instead of $r_\ell=(1\pm\epsilon)\log(\ell)$). Consequently, the same proof shows that for any such sequence the set $\{\ell| x u_{\ell}\in B_{r_{\ell}}\}$ is finite (respectively infinite) for $\sigma$-a.e. $x\in X$ if and only if the sequence $\sum_{\ell} \sigma(B_{r_\ell})$ converges (respectively diverges). That is, we show that the family of cusp neighborhoods $\mathfrak{B}=\{B_r|r>0\}$ is Borel-Cantelli for the unipotent flow.
\end{rem}


\end{document}